\documentclass[11pt,openany,leqno]{article}
\usepackage{amsmath,amsthm,amsfonts,amssymb,amscd,url}
\usepackage{graphics}

\input{xy} \xyoption{all}
\usepackage[latin1]{inputenc}

\usepackage{enumerate}
\usepackage{hyperref}
\usepackage{epsfig} 
\input{xy} \xyoption{all}

\begin{document}
\def\Diff{\text{Diff}}
\def\Max{\text{max}}
\def\P{\mathbb P}
\def\R{\mathbb R}
\def\N{\mathbb N}
\def\Z{\mathbb Z}
\def\a{{\underline a}}
\def\b{{\underline b}}
\def\n{{\underline n}}
\def\Log{\text{log}}
\def\loc{\text{loc}}
\def\inta{\text{int }}
\def\det{\text{det}}
\def\exp{\text{exp}}
\def\Re{\text{Re}}
\def\lip{\text{Lip}}
\def\leb{\text{Leb}}
\def\dom{\text{Dom}}
\def\diam{\text{diam}\:}
\def\supp{\text{supp}\:}
\newcommand{\ovfork}{{\overline{\pitchfork}}}
\newcommand{\ovforki}{{\overline{\pitchfork}_{I}}}
\newcommand{\Tfork}{{\cap\!\!\!\!^\mathrm{T}}}
\newcommand{\whforki}{{\widehat{\pitchfork}_{I}}}
\newcommand{\marginal}[1]{\marginpar{{\scriptsize {#1}}}}
\def\sR{{\mathfrak R}}
\def\sM{{\mathfrak M}}
\def\sA{{\mathfrak A}}
\def\sB{{\mathfrak B}}
\def\sY{{\mathfrak Y}}
\def\sE{{\mathfrak E}}
\def\sP{{\mathfrak P}}
\def\sG{{\mathfrak G}}
\def\sa{{\mathfrak a}}
\def\sb{{\mathfrak b}}
\def\sc{{\mathfrak c}}
\def\se{{\mathfrak e}}
\def\sg{{\mathfrak g}}
\def\sd{{\mathfrak d}}
\def\sr{{\mathfrak {sr}}}
\def\ss{{\mathfrak {s}}}
\def\spp{{\mathfrak {p}}}
\def\arr{\overleftarrow}

\newtheorem{prop}{Proposition} [section]
\newtheorem{thm}[prop] {Theorem}
\newtheorem{defi}[prop] {Definition}
\newtheorem{lemm}[prop] {Lemma}
\newtheorem{sublemm}[prop] {Sub-Lemma}
\newtheorem{cor}[prop]{Corollary}
\newtheorem{theo}{Theorem}
\newtheorem{theoprime}{Theorem}
\newtheorem{Claim}[prop]{Claim}
\newtheorem{fact}[prop]{Fact}
\newtheorem{Scholium}{Scholium}
\newtheorem{\theScholium}{\Alph{Scholium}}
 
\newtheorem{coro}[theo]{Corollary}
\newtheorem{defprop}[prop]{Definition-Proposition}
\newtheorem{propdef}[prop]{Proposition-Definition}

\newtheorem{question}[prop]{Question}
\newtheorem{conjecture}[prop]{Conjecture}

\theoremstyle{remark}
\newtheorem{exam}[prop]{Example}
\newtheorem{rema}[prop]{Remark}

\renewcommand{\thetheo}{\Alph{theo}}
\renewcommand{\thetheoprime}{\Alph{theo}$'$}

\newtheorem{propfonda}[prop]{\bf Fundamental property of the parablenders}


\title{Generic family with robustly infinitely many sinks}

\author{Pierre Berger
}

\date{\today}

\maketitle

\begin{abstract}
We show, for every $r>d\ge 0$ or $r=d\ge 2$, the existence of a Baire generic set of $C^d$-families of $C^r$-maps $(f_a)_{a\in \R^k}$ of a manifold $M$ of dimension $\ge 2$, so that for every $a$ small the map $f_a$ has infinitely many sinks. When the dimension of the manifold is greater than $3$, the generic set is formed by families of diffeomorphisms. When $M$ is the annulus, this generic set is formed by local diffeomorphisms.
This is a counter example to a conjecture of Pugh and Shub.
\end{abstract}
\section*{}
The aim of dynamical systems is to study the orbits of a given system, especially their asymptotic and statical behaviors. For this end, 
by the Birkhoff ergodic Theorem, given a dynamics $f$ of a compact manifold $M$, it suffices to study the invariant probability measures $\mu$ ($f^*\mu=\mu$) which are \emph{ergodic} (that is $\mu(A)=0$ or $1$ for every $f$-invariant set $A$).
Indeed, we know that for $\mu$-almost every point $x\in M$,   for every open subset $U$, the following property holds:
\begin{equation}\tag{$\mathcal B$}\lim_{N} \frac1N  \mathrm{Card}\{i\le N\colon f^i(x)\in U\}=\mu(U)\;.\end{equation}
In other words, all the points in a set of full $\mu$-measure have their orbits which share the same statistical behavior. The \emph{basin $B_\mu$}  of $\mu$ is the set of points in $M$ which satisfy $(\mathcal B)$ for every open subset $U$.

The measure $\mu$ is  \emph{physical} when the basin $B_\mu$ is of positive Lebesgue measure. 

Natural questions that arise are the following: for a typical dynamical systems, does Lebesgue almost every point belong to the basin of a physical measure? 
Is the number of ergodic, physical measures finite?

A positive answer to both questions would justify the mere use of statistical data to describe any typical bounded, finite-dimensional dynamics.
It would imply that the typical orbits of such systems have always a statistical behavior, and that the number of different such statistical behaviors is finite.

It is actually the gamble that many scientists do while evaluating their data.  It is also the spirit  I understand of the main global conjecture of Palis \cite{Pa95, Pa05}. 
 
Nevertheless, there are examples of dynamics without physical measure, which can be quadratic maps of the interval \cite{Jo87, HK90}, and surface flow (R. Bowen counter example).

 But these counter examples may be not typical.

Also, by the work of Newhouse \cite{Newhouse}, for every manifold $M$ of dimension $\ge 2$   and every $r\ge 2$,   there exist an open set $U\in Diff^r(M)$, and a Baire residual subset $\mathcal R\subset U$  such that every $f\in \mathcal R$ has infinitely many sinks, and so an infinite number of ergodic, physical measures.
We recall that a subset $\mathcal R\subset U$ is \emph{Baire residual} if it is a countable intersection of open and dense subsets of $U$. The example of Newhouse is then typical in a topological sense.  

Yet, there exist examples of residual sets of the real line with zero Lebesgue measure (such as the set of Liouville numbers). Hence in the measure theoretical point of view the example of Newhouse may be not typical. However, there is no canonical measure on $Diff^r(M)$. 
Nevertheless, there are several notions of prevalence which generalize the concept of "full Lebesgue measure" for Banach spaces, a panorama of them can be found in \cite{HK10}. 

In this work we will deal with the notion of \emph{typical dynamical system in the Kolmogorov sense}.

\begin{defi} Let $k\ge 0$, $d\ge 0$. A property $\mathcal P$ is valid for \emph{typical maps of $C^r(M,M)$ in the $C^d$-$k$-Kolmogorov sense}, if for all generic families $(f_a)_a\in C^d([-1,1]^k, C^r(M,M))$, for Lebesgue almost every parameter $a$, the map $f_a$ satisfies the property $\mathcal P$.   
\end{defi}

It is a conjecture of Pugh and Shub \cite[Conj. 3]{PS96} that a typical dynamical systems in the Kolmogorov sense has locally finitely many attractors (more precisely that the equivalence class of points in the chain recurrent sets are open in the chain recurrent set).

In low dimensional dynamics, this notion was also used in the Palis programme. He conjectured the (local) finiteness of the attractors for typical dynamics in the Kolmogorov sense for one dimensional dynamics and surface diffeomorphisms (see $v^{th}$-item of Main Global Conjecture in  \cite{Pa95, Pa05}, Conjecture 5 in \cite{Pa08}, and Conjecture 1 p. 150 in \cite{PT93}).

For one dimensional dynamics he moreover conjectured that Lebesgue almost every point belongs to the basin of a physical invariant probability.  This conjecture has been fully proved for one-dimensional maps displaying only one critical point: by Lyubich \cite{Ly02} for quadratic families, Avila, de Melo and Lyubich \cite{ALDM03} in the analytic case under the hypothesis of negative Schwartzian derivative and, Avila and Moreira  announced that the result holds in the complement of a set of positive (Hausdorff) codimension \cite{AM02}.
Let us mention also the result of Kozlovski, Shen and van Strien showing the density of hyperbolicity for multimodal maps \cite{KSvS07}.

The finiteness of attractors is well known to be true for dynamics which are uniformly hyperbolic. Another Palis conjecture states that homoclinic tangencies appear densely in the complement of uniformly hyperbolic dynamics for surface diffeomorphisms. The latter conjecture has been proved for $C^1$-diffeomorphisms \cite{PS00}, and in a weak sense for holomorphic maps of $\mathbb C^2$ \cite{LD13, BD14}.

That is why Conjectured 5 in \cite{Pa08} states that in the unfolding of homoclinic tangencies for generic one-dimensional
parameter families of $C^2$-surface diffeomorphisms, with total probability in the parameter line, the corresponding diffeomorphisms do not display infinitely many attractors, in particular sinks, in a sufficiently small neighborhood of the closure of the homoclinic orbit.
The latter conjecture has been proved in a weaker sense in many contexts \cite{TeYo, GoKa, BeSi, PY09, BC2, BY, BV2, Berhen, Berent}. 

In the opposite direction D. Turaev showed in a conference \cite{Tu} that in the unfolding of a dissipative homoclinic tangency for surface diffeomorphisms, there is a dense set of $C^d$-families which have infinitely many sinks for a parameter interval (depending on the family).  

In this work, we show that for surface local $C^r$-diffeomorphisms or diffeomorhisms of a $n$-manifold, $n\ge 3$, 
it is not true that a typical $C^r$-dynamical systems has finitely many attractors, in the $C^d$-$k$-Kolmogorov sense for all  
$k\ge 0$, $\infty \ge r>d\ge 0$ or $\infty > r= d\ge 2$. In particular, this gives a counter example to Pugh-Shub Conjecture for these classes of regularity.

More precisely, we show the following Theorems:
\begin{theo}\label{theo1}
For every compact surface $M$, for all $\infty \ge r>d\ge 0$ or $\infty > r= d\ge 2$, for all $k\ge 0$,
there exists an open set  $\hat U$ in $C^d(\R^k, C^r(M,M))$ and a Baire residual set $\mathcal R$ in  $\hat U$ so that for every 
$(f_a)_a\in \mathcal R$, for every $|a|\le 1$, the map $f_a$ has infinitely many sinks.
\end{theo}

\begin{theo}\label{theo2}
For every compact manifold of dimension $\ge 3$, for all $\infty \ge r>d\ge 0$ or $\infty > r= d\ge 2$, for all $k\ge 0$,
there exists an open set  $\hat U$ in $C^d(\R^k, Diff^r(M,M))$ and a Baire residual set $\mathcal R$ in  $\hat U$ so that for every 
$(f_a)_a\in \mathcal R$, for every $|a|\le 1$, the map $f_a$ has infinitely many sinks.
\end{theo}

\begin{question} 
Does Theorems \ref{theo1} and \ref{theo2} hold true for $d= \infty$ and $r=\infty$, or for $d=1$ and $r=1$?
What about the analytic case? 
\end{question}

\begin{rema} The same property holds for the sets $\hat U_\epsilon:=\{(f_{\epsilon \cdot a})_a:\; (f_a)_a\in \hat U\}$ and $\mathcal R_\epsilon:=\{(f_{\epsilon \cdot a})_a:\; (f_a)_a\in \mathcal R\}$, with $0<\epsilon <1$.

The set $U:= \{f_0; (f_a)_a\in \hat U\}$ is open, and for every $f_0\in U$,  the constant family $(f_0)_a$ belongs to the closure of $\cup_{n\ge 1} \mathcal R_{1/n}$ which is Baire residual in $\cup_{n\ge 1} \hat U_{1/n}$.
\end{rema}

\begin{Scholium}\label{Scho} For $M$, $k$, $r$ and $d$ as in Theorems \ref{theo1} and \ref{theo2}, there exists an open set  $\hat U$ in $C^d(\R^k, C^r(M,M))$ and a Baire residual set $\mathcal R$ in  $\hat U$ so that  for every $f\in U:=\{f_0:\; (f_a)_a\in \hat U\}$, the set $\mathcal R_f:= \{(f_a)_a\in \mathcal R:\; f_0=f\}$ is generic in 
$\hat U_f:= \{(f_a)_a\in \hat U:\; f_0=f\}$, and for every $0<|a|<1$, the dynamics $f_a$ carries infinitely many sinks.
\end{Scholium}
This Scholium shows that the set $\mathcal R_0\subset U$ formed by maps with infinitely many sinks is generics. Hence the set $\mathcal R:= \cup_{f\in \mathcal R_0} R_f$ is generic in $\hat U$. As the set of maps in $U$ carrying infinitely many sinks is Baire residual (see Remark \ref{BlenderNewhouse}), this Scholium implies Theorems \ref{theo1} and \ref{theo2}.

\begin{question} 
Does there exist an open set $U$ of $C^r$-diffeomorphisms, in which a typical diffeomorphisms has infinitely many sinks in the $C^d$-$k$-Kolmogorov sense?
In other words, does there exist a generic set $\mathcal R$ in $C^d(\R^k , U)$ so that for every $(f_a)_a\in \mathcal R$, the map $f_a$ has infinitely many sinks for Lebesgue almost every parameter $a$?  
\end{question}

From Newhouse works, the answer is positive for $k=0$ or $d=0$. The presented theorems do not answer to this question for $k>0$ and $d>0$ since we can imagine that there exist open sets $\hat U, \hat U'$ in $C^d([-1,1]^k, Diff^r(M))$, and Baire residual set $\mathcal R\subset \hat U, \mathcal R'\subset \hat U'$, so that:
\begin{itemize} 
\item  $\{f_0: (f_a)_a\in \hat U\}=\{f'_0: (f'_a)_a\in \hat U'\}$, 
\item for every $(f_a)_a\in \mathcal R$, for every $(f_a')\in \mathcal R'$, for Lebesgue a.e. $a$, the map $f_a$ has infinitely many sinks, whereas $f'_a$ has only finitely many sinks.
\end{itemize}

The proof of Newhouse phenomena was done by constructing a hyperbolic set $K$ which has robust homoclinic tangencies. This means that for every perturbation of the dynamics, there are two points $P,Q$ in the hyperbolic continuation of $K$ such that the stable manifold of $P$ is tangent to the unstable manifold of $Q$. The construction of Newhouse worked for  surface diffeomorphisms and in the $C^r$-topology, $r\ge 2$. These robust homoclinic tangencies are implied by a result which gives sufficient conditions for two (dynamical) Cantor sets of the real line to have stable intersections (for perturbation in the space of such dynamical Cantor sets). These conditions are based on the geometry of these Cantor sets, via a concept called thickness and afterward \cite{GuYo}, via an optimal inequality on their Hausdorff dimensions.
This construction was generalized in \cite{PaVi} to manifolds of higher dimension. 

 Another construction of Newhouse was done for $C^r$-diffeomorphisms of $n$-manifold, for $n\ge 3$ and $r\ge 1$ \cite{New80} (see also \cite{Asa08}). Latter, Bonatti-Diaz constructed an example of hyperbolic set $B$ of $\R^3$ with $C^1$-robust homoclinic tangencies, by constructing a hyperbolic set with special properties : the \emph{blender}. This implied the existence of a Baire residual set in an open set of $Diff^1(M^3)$ for which all diffeomorphisms have infinitely many sinks \cite{BD99, DNP06}.

We will prove Theorems \ref{theo1} and \ref{theo2} by introducing two notions. The first is the \emph{$C^d$-paratagency}, which is a (quadratic) tangency between two curves of $\R^2$ which is unfolded degenerately up to the order $d$. The second notion is the \emph{parablender} which produces robust homoclinic $C^d$-paratangencies. We will then use some perturbations techniques to prove the following result on local diffeomorphisms of the annulus $\mathbb A=\mathbb S^1 \times [-1,1]$.

\begin{theo}\label{main}
For all $\infty\ge r>d\ge 0$ or $\infty > r= d\ge 2$, for all $k\ge 0$, there exists a nonempty open set $U$ in $C^r (\mathbb A,\mathbb A)$ formed by local diffeomorphisms, such that $f(\mathbb A)\Subset \mathbb A$, there exists an open set  $\hat U$ in $C^d(\R^k, C^r(\mathbb A,\mathbb A))$, and a family of open dense subsets $(\hat U_N)_N$ in $\hat U$  satisfying that:
\begin{itemize}
\item  for every $(f_a)_a\in \hat{U}_N$, for every $a\in[-1,1]^k\setminus[-1/N,1/N]^k$, the map $f_a$ has $N$-sinks,
\item  for every $f\in U$, the set $\{(f_a)_a\in U_N:\; f_0=f\}$ is non-empty, open and dense in $\{(f_a)_a\in \hat U:\; f_0=f\}$.\end{itemize}\end{theo}

\begin{rema} For every $r\ge 2$, with $d'= \max(2,d)$, the open set  $U$ (resp. $U_N$ and $\hat U$) is the intersection of open subsets of  $C^{d'} (\mathbb A,\mathbb A)$ (resp. $C^d(\R^k, C^{d'}(\mathbb A,\mathbb A))$) with $C^r(\mathbb A,\mathbb A)$ (resp. $C^d(\R^k, C^{r}(\mathbb A,\mathbb A))$).
\end{rema}

As every compact surface $M$ contains an annulus $\mathbb A$, we notice that it is possible to extend any local diffeomorphism $f$ of the annulus to the whole surface $M$. Hence Theorem \ref{main} implies Scholium \ref{Scho} and so Theorem \ref{theo1}. Nevertheless the extension cannot be a local diffeomorphism if $M$ is not the torus.  

To get Theorem \ref{theo2}, the maps of $U$  will be lifted as maps of a solenoid of the $n$-balls, $n\ge 3$.  Although this construction is simple to implement, it does not allow us to carry the case $r=\infty$ for diffeomorphisms (contrarily to Theorems \ref{theo1} and \ref{main}). However a canonical generalization of the results proved in this paper to higher dimension would carry this case.


\thanks{I thanks the referees for their advices and comments. I am very grateful to Enrique Pujals and Sylvain Crovisier for important conversations and their comments on the first version of this work. I am thankful to Jean-Christophe Yoccoz for his encouragements. This research was partially supported by the Balzan project of J. Palis, the French-Brazilian network and the project BRNUH of Universit\'e Sorbonne Paris Cit\'e.}

\section*{Plan of the proof}
We describe now the content of the rest of this paper. There are basically three steps.

\paragraph{Paratangency and Parablender}
In section \ref{Paratangencyetprelim}, we start by defining the  topologies involved and their notations. Then we recall the definition of hyperbolic sets for local diffeomorphisms, and how we define the stable and unstable manifolds of their points. Then we state Proposition \ref{cartecool4cpct} claiming that for a $C^d$-family of $C^r$-local diffeomorphisms, the family of local stable and unstable $C^r$-manifolds of a given point is also of class $C^d$. The proof of this result is given in  appendix \ref{section:proofhypprel}.

Then we introduce a key concept in the work: the notion of \emph{homoclinic $C^d$-paratangency} for a hyperbolic point $\Omega$. For families of $C^r$-diffeomorhisms $(f_a)_a$, it is a  homoclinic tangency at a certain parameter $a_0$, which unfolds degenerately up to the order $d$, although the tangency is quadratic if $r\ge 2$. For instance, if $(g_a)_{a\in\R }$ has a homoclinic tangency at $a=a_0$, then with  $f_a:=g_{a_0 +(a-a_0)^{d+1}}$, the family $(f_a)_a$ has a homoclinic $C^d$-paratangency at $a=a_0$. This concept will need $r\ge d\ge 2$ or $r>d\ge 0$ to be well defined. We will always make this assumption afterward.

In section \ref{section:parablender}, we introduce the main concept of this work : the parablenders. 
The \emph{blenders} are hyperbolic compact set well known to produce robust homoclinic tangencies. We recall their definition in the context of local diffeomorphisms of surface. They have robustly the following fundamental property: for every curve $\Gamma$ in a  $C^r$-open set, there is a local unstable manifold of the blender which is tangent to $\Gamma$.

Along the stable foliation of the blender, a toy model of the dynamics is the iterated function system generated by the two following contracting maps:
\[x\mapsto \frac 23 x+\frac13,\quad x\mapsto \frac 23 x-\frac13\]
Its limit set is $C^1$-robustly an interval close to $[-1,1]$.

Contrarily to blenders, the parablenders are only defined for families of local diffeomorphisms. They enable to produce not only robust homoclinic tangencies, but also robust homoclinic $C^d$-paratangencies. Given a family of surface local diffeomorphisms $(f_a)_a$, a \emph{parablender} is a family of blenders $(K_a)_a$ which satisfies the following fundamental property (see Prop. \ref{Ml2} and \ref{Ml2k}): for every $C^d$-family of $C^r$-curves $(\Gamma_a)_a$ in an  open set, for every $a_0$, there exists  a local unstable manifold $\Gamma^u_a$ of $K_a$, so that   $(\Gamma^u_a)_a$ and $(\Gamma_a)_a$ are paratangent at $a=a_0$: for $a$ close to $a_0$, the curves $\Gamma_a$ and $\Gamma_a^u$ are tangent up to a deformation smaller than $o(|a-a_0|^d)$. This fundamental property is shown to be robust for every perturbation of the family $(f_a)_a$ in the given example of parablender.

Along the stable foliation of the blender, a toy model of the dynamics is the iterated function system generated by the following contracting maps among $\delta= (\delta_i)_{i=1}^d \in \{-1,1\}^d$:
\[x\mapsto \frac 23 x+\frac13+\epsilon P_\delta(a),\quad x\mapsto \frac 23 x-\frac13+\epsilon P_\delta(a)\]
with  $P_\delta (X)= \sum_{i=1}^d \delta_i X^i$. There are $2^{d+1}$ contracting maps. The points of the limit set have hyperbolic continuities whose $d$-first  derivatives w.r.t. $a$  form $C^1$-robustly an open set of the $d$-jet space, for every parameter $a$ close to $0$.

\paragraph{Sink creation from homoclinic paratangency}
In section \ref{section:creasink}, we first recall that one can perturb a dissipative homoclinic tangency to create a new sink. 
This result is generalized for parameter family exhibiting a dissipative homoclinic $C^d$-paratangency: a $C^d$-perturbation of the family of $C^r$-dynamics produces a large parameter set for which there is a new sink.  
 
More precisely, Lemma \ref{robusthomoclinictangency} states that given a $C^d$-family of $C^r$-local diffeomorphisms $(g_a)_a$ having a homoclinic paratangency of a hyperbolic point $\Omega_a$, 
for every $\alpha>0$ small enough (depending on technical bounds), there exist an arbitrarily small perturbation $(g_a')_a$ of $(g_a)_a$, so that for every $\|a-a_0\|<\alpha$, the point $\Omega_a$ has a homoclinic tangency.
Note that the conclusion of this Lemma is wrong for a non-degenerate homoclinic unfolding: none perturbations would give a (locally) stationary homoclinic tangency.
Then, if $\Omega_a$ is strongly dissipative, we show for Proposition \ref{sinkcreation} that a small perturbation of $(g'_a)_a$ creates a new periodic sink for every $\|a-a_0\|<\alpha$. The main difficulty is postpone to Claim \ref{unpuit}: it shows that the perturbation of   $(g'_a)_a$ is small for the $C^d$-topology of local diffeomorphisms families.

In order to apply these results to prove Theorem \ref{main} with the parablender, we generalize Proposition \ref{sinkcreation} to  the case $C^d$-paratangency between two points of a same basic set. 
To formalize this we define the notion of \emph{quasi-homoclinic $C^d$-paratangency}. Such a quasi-homoclinic $C^d$-paratangencies are easily perturb to homoclinic $C^d$-paratangencies.  
Hence Proposition \ref{quasisinkcreation} states that given a $C^d$-family of local diffeomorphisms $(g_a)_a$ exhibiting a quasi-homoclinic $C^d$-paratangency of a (strongly) dissipative hyperbolic fixed point at $a_0$, for every $\alpha>0$ sufficiently small (depending on technical bounds), there exists an arbitrarily  small perturbation of $(g'_a)_a$ with a new periodic sink for every $\|a-a_0\|<\alpha$.

\paragraph{Lattice of parameter sets with sinks}
In section \ref{section:proofmain}, we prove Theorem \ref{main} at the neighborhood of dynamics that first appear in \cite{DNP06} in the diffeomorphisms case.
We start with a family of maps $(f_a)_a$ described in  section \ref{section:parablender}, which has a parablender.  We deform these dynamics to a family of maps $(g_a)_a$ so that $(\Gamma_a)_a$ (involved in the fundamental property of the parablender) is a segment of the stable manifold of a (strongly) dissipative hyperbolic fixed point $\Omega_a$, and also both $\Omega_a$ and the blender $K_a$ are included in a transitive hyperbolic set. Then by the fundamental property of the parablender, there exists a neighborhood $V_a$ which does not depend on  $(g_a)_a$ in an open set of $C^d$-families of $C^r$-maps, so that for every $a_0\in V_a$, the hyperbolic fixed point $(\Omega_a)_a$ has a homoclinic $C^d$-paratangency. Hence for every $\alpha>0$ small enough, Proposition  \ref{quasisinkcreation}  gives the existence for every $a_0\in V_a$, of a perturbation $(g'_a)_a$ of $(g_a)_a$, such that for every $a$ which is $\alpha$-distant to $a_0$, the map $g_a$ has a new periodic sink. 
The main difficulties of this part is to verify the technical bound hypothesis of Proposition \ref{quasisinkcreation} in order that $\alpha$ does not depend on $a$. 

Furthermore Proposition \ref{quasisinkcreation} states that $(g'_a)_a$ can be chosen equal to $g_a$ whenever $a$ is $2\alpha$-distant to $a_0$.
Hence we can apply this proposition for $a_0$ in the lattice $(4\alpha \mathbb Z^k\setminus \{0\}) \cap V_a$. By hyperbolic continuation, the new sinks created for $a$ in the $\alpha$-neighborhood of this lattice persists for an open set of local diffeomorphisms families. Hence we are able to apply this inductively for $\alpha=1/n$, with $n$ large. 
Moreover by taking $2^k$-copies of this dynamics, and we can apply this argument at a special translation of the lattice $4\alpha \mathbb Z^k\setminus \{0\}$ for each copy, so that the union of the $\alpha$-neighborhoods of these lattices contains $V_a$. This will prove  Theorem \ref{main}.

In section \ref{proof:theo2}, we prove Theorem \ref{theo2}. In order to do so,  we lift the previous construction to make the dynamics invertible, by basically using Smale solenoid construction. The construction will produce a partially hyperbolic map with a strong stable direction, and the holonomy along the strong stable manifolds will define a semi-conjugacy with a surface local diffeomorphism as in the proof of Theorem \ref{main}.

\section{Paratangencies, classical results of hyperbolic dynamics and notations}\label{Paratangencyetprelim}
\paragraph{Notation for the involved topologies}

In this work we will only work with regularity given by the integers $r>d\ge 0$. Given $n,m\ge 1$, the \emph{topology} on  $C^r(\R^n,\R^m)$ is given by the norm:
\[\|g\|_{C^r}:= \sup_{0\le i\le r, \, z\in \R^n} \|D^i_z g\|.\]
A map $g\in C^r(\R^n,\R^m)$ is bounded if its norm is finite. The space of bounded $C^r$-maps endowed with this norm is complete. 
 
The space  $C^d(\R^k, C^r(\R^n,\R^m))$ is formed by the families $(g_a)_a$ of maps $g_a\in C^r(\R^n,\R^m)$ so that the derivatives 
$\partial_ a^j  \partial_z^ig_a(z)$ are well defined for every $i\le r$ and $j\le d$ and depend continuously on $(a,z)\in\R^k\times \R^n$. The \emph{topology} on $C^d(\R^k, C^r(\R^n,\R^m))$  is given by the norm:
\[\|(g_a)_a\|_{C^{d,r}}:= \sup_{0\le i\le r, \, 0\le j\le d,\, z\in \R^n} \|\partial^j_a\partial^i_z g_a(z)\|\;.\]
A family $(g_a)_a$ is \emph{bounded} if its norm is finite. The space of bounded $C^d(\R^k, C^r(\R^n,\R^m))$-maps endowed with this norm is complete.
We define similarly the \emph{topology} on $C^d(\R^k,C^r(M,N))$ where $M$, $N$ are Riemannian manifolds.

The topology on $C^\infty (\R^n,\R^m)$ (resp. $C^d(\R^k, C^\infty (\R^n,\R^m))$ ) is the union of the topologies induced by $C^r(\R^n,\R^m)$ (resp. $C^d(\R^k, C^r (\R^n,\R^m))$ ) among $r\ge 0$.

\paragraph{Notation :}
For the sake of simplicity, we will write that \emph{$(g_a)_a$ is of class $C^{d,r}$} to say that $(g_a)_a$ is a $C^d$-family of $C^r$-functions.

The distance between two submanifolds which are the graphs of functions in  $C^r(\R^n,\R^m)$ (resp. $C^d(\R^k, C^r(\R^n,\R^m))$) is the distance between these functions given by the norm defined above.

\paragraph{Paratangency}

The following definition is useful to define the concept of homoclinic paratangency, which will be crucial in the proof of the main theorem.

\begin{defi}[$C^d$-Paratangency]
Let $r> d\ge  0$ or $r= d\ge 2$, and let $(\Gamma_a)_{a\in \R^k}$ and $(W_a)_{a\in \R^k}$ be two $C^{d}$-families  of immersed $C^{r}$-curves in $\R^2$.
We say  that they are $C^d$-\emph{paratangent} at the parameter $a=0$ and the points $P\in \Gamma_0\cap W_0$, if the following property holds.
If $r=1$ and so $d=0$, this means that $\Gamma_0$ and $W_0$ are tangent at $P$. If $r\ge 2$, this means that there exists a $C^d$-family $(\phi_a)_a$ of $C^\infty$-diffeomorphisms of $\R^2$ such that 
\begin{itemize}
\item[$(o)$] The curve $\Gamma_0$ and $ W_0$ are tangent at $P$, and the tangency is quadratic (i.e. the curvatures of $ W_0$ and $\Gamma_0$ at $P$ are different),
\item[$(i)$] $\phi_a(W_a)$ is tangent to $\Gamma_a$ at a point nearby $P$,
\item[$(ii)$] at $a=0$, $\partial_a^i \phi_a=0\in C^r(\R^2)$, for every $0\le i\le d$.
\end{itemize}
\end{defi}


\begin{rema}\label{Remprametrizationtangent} The definition of $C^d$-paratangency does not vary by coordinate change, and so might be defined for curves families of any surface.
Indeed, if $(W_a)_a$ and $(\Gamma_a)_a$ are $C^d$-paratangent at $P$, there exists a $C^d$-family of diffeomorphisms $(\phi_a)_a$ satisfying $(i)$ and $(ii)$.

Let $(\Psi_a)_a$ be a $C^d$-family of $C^r$-diffeomorphisms. To show that $(\Psi_a(W_a))_a$ and $(\Psi_a(\Gamma_a))_a$ are $C^d$-paratangent at $\Psi_a(P)$ and $a=0$, it suffices to consider the $C^d$-family of diffeomorphisms $(\phi'_a)_a$ defined by $\phi'_a:= \Psi_a\circ \phi_a\circ \Psi_a^{-1}$, which satisfies $(i)$ and $(ii)$ for the families $(\Psi_a(W_a))_a$ and $(\Psi_a(\Gamma_a))_a$ .

\end{rema}

The existence of a family $(\phi_a)_a$ of class $C^d$ might sound difficult to implement. It is actually rather easy as shown by the following fact. 

\begin{fact}\label{prametrizationtangent}
Let $\infty \ge r\ge d\ge 1$ with  $r\ge 2$, and let $(\Gamma_a)_{a\in \R^k}$ and $(W_a)_{a\in \R^k}$ be two $C^{d}$-families  of immersed $C^{r}$-curves in $\R^2$.
 By coordinate change, we assume that $P=0$ and that a  neighborhood of $P$ in $\Gamma_a$ is a segment of  $\R\times \{0\}$. 

Suppose that there exists $P\in \Gamma_0\cap W_0$ at which $\Gamma_0$ and $W_0$ are tangent but with curvatures different.
 Then there exists a unique function $a\mapsto \tau(a)$ defined on a neighborhood of $0$ so that: 
\begin{enumerate}[$(i)$]
\item $\tau$ is of the form $\tau(a)=(0,\eta(a))$ for every small $a$, 
\item $W_a+\tau(a)$ is tangent to $\Gamma_a$ nearby $P$.
\item the tangency point is a $C^{\min(d,r-1)}$-function of $a$ and is uniquely defined at the neighborhood of $P$.
\end{enumerate}
Moreover $\tau$ is of class $C^d$.

 Also $(\Gamma_a)_{a\in \R^k}$ and $(W_a)_{a\in \R^k}$ are $C^d$-paratangent if and only if $\partial_a^i\tau(a)=0$ for every $0\le i\le d$.
\end{fact}

\begin{proof}[Proof of Fact \ref{prametrizationtangent}]
Since the tangency is quadratic, 
for $a$ small, the intersection of a neighborhood of $P$ with  $\Gamma_a$ is a graph of a $C^r$-function $\gamma_a$, which satisfies $\gamma_0(P)=0$, $\partial_x \gamma_0(P)=0$ and $\partial^2_x \gamma_0(P)\not =0$.
To prove the first part of the claim, it suffices to show that with $c_a$ the critical point of $\gamma_a$ near $P$, the map  $\eta\colon a\mapsto \gamma_a(c_a)$ is of class $C^d$. To prove such a regularity, we solve the following equation  implicitly:
\[\Psi(x,a)=0, \quad\text{with }\Psi\colon \R^2\ni (a,x)\mapsto \partial_x \gamma_a(x)\in \R.\]
The implicit function theorem implies that $a\mapsto c_a$ is of class $C^{\min(d,r-1)}$, and so at least of class $C^1$. Hence we can look at the following derivative:
\[\partial_a (\gamma_a(c_a))= (\partial_a  \gamma_a)(c_a)+ (\partial_x \gamma_a)(c_a)\cdot \partial_a c_a\]
As $c_a$ is the critical point of $\gamma_a$, the smoothness of $\partial_a (\gamma_a(c_a))$ is the same as the one of $(\partial_a  \gamma_a)(c_a)$ which is $C^{\min(d,r-1,d-1)}$. As $r\ge d$, it comes that $\eta\colon a\mapsto \gamma_a(c_a)$ is of class $C^d$.



Taking $\phi_a\colon z\mapsto z+\tau(a)$, clearly  $(\Gamma_a)_{a\in \R^k}$ and $(W_a)_{a\in \R^k}$ are $C^d$-paratangent at $a=0$ if $\partial_a^i\tau(a)=0$ for every $0\le i\le d$.

Now suppose that  $(\Gamma_a)_{a\in \R^k}$ and $(W_a)_{a\in \R^k}$ are $C^d$-paratangent at $a=0$ and $P$, via a map $(\phi_a)_a$ by satisfying $(i)$ and $(ii)$. 
Hence for every $a$ small, there exists $Q_a\in \Gamma_a$ which is sent by $\phi_a^{-1}$ into $\R\times \{0\}$. We notice that $Q_a$ has its $y$-coordinate of modulus at least $\|\tau(a)\|$. Together with the $C^d$-paratangency condition, this gives $\|\tau(a)\|\le \|\phi_a\|_{C^r}\le o(\|a\|^d)$, and so $\tau(a)$ satisfies $\partial_a^i \tau(a)=0$ for every $0\le i\le d$.
\end{proof}

\paragraph{Hyperbolic dynamics preliminaries}
Let $M$ be a manifold. A subset $K\subset M$ is left \emph{invariant} by an endomorphim $f\in C^1(M,M)$  if $f(K)=K$. 

When $f$ is a diffeomorphism, the invariant compact set $K\subset M $ is \emph{hyperbolic} if there exists a $Df$-invariant splitting  $TM|K= E^s\oplus E^u$ so that $E^s $ is contracted by $Df$ and $E^u$ is expanded by $Df$:
\[\exists  \lambda<1, \; C>0,\; \forall k\in K,\;\forall n\ge 0,\;   \left\|{Df^n|E^s_k}\right\|\le C\lambda^n\text{ and }  \left\|{(Df^{n}|E^u_k)^{-1}}\right\|\le C\lambda^n.\]

When $f$ is a local diffeomorphism, we shall consider the inverse limit $\overleftarrow {K}_f$ of $K$: 
\[\overleftarrow {K}_f= \{(k_i)_i\in K^\mathbb Z:\; f(k_i)=k_{i+1},\; \forall i\in \mathbb Z\}.\]
It is a compact space for the topology induced by the product ones of $ K^\mathbb Z$. The dynamics induced by $f$ on $\overleftarrow {K}_f$ is the shift $\overleftarrow {f}$; it is invertible. Let $\pi\colon \overleftarrow {K}_f\to K$ be the zero-coordinate projection. Let $\pi^*TM$ be the bundle over $\overleftarrow {K}_f$ whose fiber at $\underline k$ is $T_{\pi(\underline k)} M$.
 The map $Df$ acts canonically on $\pi^*TM$ as the map $\overleftarrow {f}$ on the basis and as the linear map $D_{\pi(\underline k)}f$ on the fiber of $\underline k\in \overleftarrow {K}_f$.
  
The compact set $K$ (or $\overleftarrow {K}_f$) is \emph{hyperbolic} if there exists a $Df$-invariant splitting  $\pi^*TM= E^s\oplus E^u$ so that $E^s_{\underline k} $ is contracted by
 $D_{\pi(\underline k)}f$ and $E^u_{\underline k}$ is expanded by $D_{\pi(\underline k)}f$:
\[\exists  \lambda<1, \; C>0,\; \forall \underline k \in \overleftarrow K,\;\forall n\ge 0,\;   \left\|{Df^n|E^s_{\underline k}}\right\|\le C\lambda^n\text{ and }  \left\|{(Df^{n}|E^u_{\underline k})^{-1}}\right\|\le C\lambda^n.\]
Actually the definition of hyperbolicity for local diffeomorphisms is consistent with the definition of hyperbolicity for diffeomorphisms when the dynamics is invertible.
 Here is a useful structural stability result:
\begin{thm}[Przytycki]\label{Przy}
Let $\overleftarrow K$  be a hyperbolic set for a $C^1$-local diffeomorphism $f$ of $M$.
Then for every $f'$ $C^1$-close to $f$, there exists a continuous map $i_{f'}\colon \overleftarrow K_f\to M$ which is $C^0$-close to $\pi$ and so that:
\begin{itemize}
\item $i_{f'}\circ \overleftarrow f = f'\circ i_{f'}$,
\item $K_{f'}:=i_{f'}( \overleftarrow K_f)$ is hyperbolic for $f'$.
\end{itemize}
\end{thm}
This Theorem will be reproved for Proposition \ref{cartecool4cpct}.

For every $\underline k\in \overleftarrow K$ we define the unstable manifold by:
\[\overleftarrow W^u(\underline k;f)= \{\underline k'\in \overleftarrow K:\; d(\overleftarrow {f}^{-n}(\underline k), \overleftarrow {f}^{-n}(\underline k'))\to 0\}\;.\]
As $f$ is a local diffeomorphism, the set $\overleftarrow W^u(\underline k;f)$ has a canonical structure of manifold which is immersed by $\pi$, see \cite{BR13}. For every $\eta>0$, let $\overleftarrow W^u_\eta(\underline k;f)$ be the component of $\underline k$ in $\overleftarrow W^u(\underline k;f)\cap \pi^{-1}(B(\pi(k),\eta))$. For $\eta$ small enough, the manifold $\overleftarrow W^u_\eta(\underline k;f)$ is embedded by $\pi$ onto a $C^r$-submanifold $W^u_\eta(\underline k;f)$. The manifolds $(W^u_\eta(\underline k;f))_{\underline k\in \overleftarrow K}$ depend continuously on $\underline k$ for the $C^r$-topology. 

It is well known that for $\underline k, \underline k'\in \overleftarrow K$ such that $\pi(\underline k)=\pi(\underline k')$, the local unstable manifolds  
$W^u_\eta(\underline k;f)$ and $W^u_\eta(\underline k';f)$ are not necessarily equal.

Nevertheless, when $p\in K$ is a fixed point, we denote by $W^u_\eta(p;f) :=W^u_\eta(\underline p;f)$ and $W^u(p;f) :=W^u(\underline p;f)$, with $\underline p\in \overleftarrow K$ the constant sequence $(p)_{n\in \mathbb Z}$.  

For every $k\in K$ we define the stable manifold by:
\[W^s(k;f)=\{k'\in M:\; d({f}^{n}(k), {f}^{n}(k'))\to 0\}\]
As $f$ is a local diffeomorphism, the set $W^s(k;f)$ has a canonical structure of manifold which is injectively immersed in $M$. For every $\eta>0$, let $W^s_\eta(k;f)$ be the component of $k$ in $W^s(k;f)\cap B(k,\eta)$. 
For $\eta$ small enough, the manifold $W^s_\eta(k;f)$ is properly embedded. The manifolds $(W^s_\eta( k;f))_{ k\in  K}$ depend continuously on $ k\in K$ for the $C^r$-topology.

Whenever it will be convenient, for every $\underline k\in \overleftarrow K$, we will denote $W^s(\underline k;f):= W^s(\pi(\underline k);f)$.

To simplify the analysis of the tangency unfolding, it is useful to work in special coordinate given by the two following Propositions. 
\begin{prop}\label{cartecool}
Let $k,r,  d\ge 0$, let  $(\Omega_a)_a$ be a continuous family hyperbolic fixed points for $(f_a)_a\in C^d(\R^k, Diff^r(\R^n))$ respectively.
Then there exists $\eta>0$, and a $C^{d}$-family of $C^r$-charts $(\phi_a)_a$ of $\eta$-neighborhoods of $(\Omega_a)_a$ so that:
\[\phi_a(\Omega_a)=0,\; \phi_a(W^u_\eta(\Omega_a))=(-\eta,\eta)\times \{0\},\; \phi_a(W^s_\eta(\Omega_a))=\{0\}\times (-\eta,\eta).\]
\end{prop}
We will prove this proposition in appendix \ref{Propositioncartecoolproof}.

Let $K$ be a hyperbolic set for a $C^r$-local diffeomorphism $f$, with $M$ a compact manifold. Let $(f_a)_a$ be a continuous family containing $f_a$, small enough so that for every $a$, the map $i_{f_a}\in  C^0(\overleftarrow K, M)$, given by Theorem \ref{Przy} is well defined. For $\underline k\in \overleftarrow K$ and $a$, we put $\underline k_a:= i_{f_a}(\underline k)$.

\begin{prop}\label{cartecool4cpct}
Let $ r\ge 1$, $k, d\ge 0$, let $M$ be a manifold. Suppose that $(f_a)_a$ is a $C^d$-family of $C^r$-local diffeomorphisms $f_a$ which leaves invariant a compact subset $K_{f_a}$. 
Then there exists $\eta>0$, and for every $\underline k\in \overleftarrow K_{f_0}$, a $C^{
d, r}$-family of charts $(\phi_a(\underline k))_a$ of $\eta$-neighborhoods of $(\underline k_a)_a$ so that:
\[\phi_a(\underline k_a)=0,\; \phi_a(W^u_\eta(\underline k_a; f_a))=(-\eta,\eta)\times \{0\},\; \phi_a(W^
s_\eta(\underline k; f_a))=\{0\}\times (-\eta,\eta).\]
Moreover the families $(\phi_a(\underline k))_a$ depend continuously on $\underline k\in \overleftarrow K_{f_0}$.
\end{prop}
We will prove this proposition in appendix \ref{Propositioncartecoolproof}.

The following Lemma is a generalization of the classical inclination Lemma (for a given diffeomorphism) to the context of $C^d$-families of local-Diffeomorphisms. 
\begin{lemm}[Inclination Lemma]\label{inclination}
Let $ r\ge 1$, $k, d\ge 0$, let $U\Subset \R^m$. Suppose that $(f_a)_a$ is a $C^d$-family of $C^r$-local diffeomorphisms $f_a$ of $U$ which leaves invariant a compact subset $K_{f_a}$.  Let $\underline k\in \overleftarrow K_{f_0}$ and $(\Gamma_a)_a$ a $C^{d,r}$-family of manifolds of the same dimension as $W^u_{loc}(\underline k; f_a)$ and so that 
$(\Gamma_a)_a$ intersects transversally $(W^s_{loc}(\underline k; f_a))_a$ at a $C^d$-curves of points $(z_a)_a$.
Then for every $\epsilon>0$, for every $n$ large, the $\epsilon$-neighborhood of $f^n_a(z_a)$ in $f^n(\Gamma_a)$ is a submanifold $\Gamma_a^n$ which is $C^r$-close to $W^u_{\epsilon}(\overleftarrow f_0^n(\underline k); f_a)$, and $(\Gamma_a^n)_a$ is $C^{d,r}$-close to $(W^u_{\epsilon}(\overleftarrow f_0^n(\underline k); f_a))_a$.
\end{lemm}
\begin{proof}
We recall that for local diffeomorphisms, the unstable manifolds may overlap. Hence we can extend $f_a$ to a larger open set $U'\supset U$ containing a hyperbolic point $H_a$ so that $\Gamma_a$ is a slice of an unstable manifold $W^u(H_a; f_a)$ and so that the preimages $(z^i_a)_{i\le -1}$ of $z_a$ along $W^u(H_a; f_a)$ are away of $W^u_{loc}(K_{f_a}; f_a)$. 

Then we consider the compact set $\hat K_{f_a}:= K_{f_a}\cup \{z^i_a:\; {n<0}\}\cup \{f_a^n(z_a):\; n\ge 0\}$. We notice that  $\hat K_{f_a}$ is a hyperbolic compact set. Also  
$\Gamma_a^n= W^s_\epsilon (z^n_a; f_a)$. Thus the Lemma follows from Proposition \ref{cartecool4cpct}.
\end{proof}

\paragraph{Homoclinic paratangency}
Let $f$ be a local $C^r$-diffeomorphism, $r\ge 1$, of an open set $U$ of $\R^2$. A hyperbolic fixed point $\Omega$ has a \emph{homoclinic tangency} if its unstable manifold $W^s(\Omega)$ and $W^u(\Omega)$ are tangent at a certain point $Q$. When $r\ge 2$, the tangency is \emph{quadratic} if they are tangent ($T_QW^s(\Omega;f)=T_QW^u(\Omega;f)$) but their curvatures at $Q$ are not equal ($T^2_QW^s(\Omega;f)\not=T^2_QW^u(\Omega;f)$).

It is well known that given a dissipative homoclinic tangency, a perturbation of the dynamics produces a new periodic sink (see Prop. \ref{pointapuit}). As Theorem \ref{main} deals with parameter families, we shall regard the families formed by local stable and unstable manifolds of a hyperbolic fixed point. Instead of looking for a mere homoclinic tangency at a certain parameter, we will look for a \emph{homoclinic paratangency}. For they imply the existence of a perturbation of the family having a new periodic sink which persists for a large open parameter set (see Prop. \ref{sinkcreation}).

Let $\infty \ge r> d\ge 0$ or $r= d\ge 2$, let $M$ be a surface and let  $(\Omega_a)_a$ be a continuous family hyperbolic fixed points for a $C^{d}$-family $(f_a)_a$ of $C^r$-local diffeomorphisms. 
By Proposition  \ref{cartecool}, there exists $\eta>0$ so that the families $(W^s_\eta(\Omega_a; f_a))_a$ and  $(W^u_\eta(\Omega_a; f_a))_a$ are of class $C^{d,r}$. By looking at their images by iterates of the dynamics, it comes that $(W^s(\Omega_a; f_a))_a$ and  $(W^u(\Omega_a; f_a))_a$ are of class $C^{d,r}$.
\begin{defi}[Homoclinic paratangency] The family of hyperbolic points $(\Omega_a)_a$ has a \emph{homoclinic $C^d$-paratangency} at $a=0$ and $Q_0\in W^s(\Omega_0; f_0)\cap W^u(\Omega_0; f_0)$ if $(W^s(\Omega_a; f_a))_a$ and  $(W^u(\Omega_a; f_a))_a$ are $C^d$-paratangent at $a=0$ and $Q_0$. In particular, if $r\ge 2$, there is a quadratic homoclinic tangency at $a=0$ and $Q_0$.
\end{defi}

The concept of homoclinic paratangency is  never generic for $d\ge 1$, since in contradiction with the notion of non-degenerate unfolding which is generic. Let us compare the two definitions. 

Let $r> d\ge 1$ or $r\ge d\ge 2$, let $M$ be a surface and let  $(\Omega_a)_a$ be a continuous family of hyperbolic fixed points for a $C^{d}$-family $(f_a)_a$ of $C^r$-local diffeomorphisms. By Proposition \ref{cartecool}, seen via a $C^{d,r}$-family of $C^r$-charts on  $D:= (-\eta,\eta)^2$, we can assume that for every $a$:
\[\Omega_a=0,\quad W^u_\eta(\Omega_a)=(-\eta,\eta)\times \{0\},\quad \text{and}\quad W^s_\eta(\Omega_a)=\{0\}\times (-\eta,\eta).\]

The point $\Omega_0$ has a quadratic homoclinic tangency if an iterate $f^N_0$ sends a point $P_0=(p_0,0)\in D$ to a point $Q_0=(0,q_0)\in D$,
and for every $a$ small, and there exists a neighborhood $D_P$ of $P$ in $D$ sent by $f_0^N$ into $D$, so that  $f_0^N|D_P$ has the following form:
\[ (p_0+x, y)\in D_P\mapsto  (0, q_0)+ (A_0(x, y), B_0(x, y))\in D\;,\]      
where $(A_0,B_0)$ is of class $C^r$ and satisfies 
\begin{equation} \tag{$\mathcal T$} 0=A_0(0)=B_0(0)=\partial_x A_0(0) \quad\text{and} \quad \partial^2_x A_0(0)\not=0\;.\end{equation}

We can keep the same form for $f_a|D_P$ with a $C^{d,r}$-family $(A_a,B_a)_a$:
\[ (p_0+x, y)\in D_P\mapsto  (0, q_0)+ (A_a(x, y), B_a(x, y))\in D\; .\]  

When $a$ varies, by $(\mathcal T)$ and Fact \ref{prametrizationtangent}, there exists a unique critical point $c_a$ of $x\mapsto A_a(x,0)$ nearby $0$, for every small $a$. In particular, it holds $c_0=0$. Moreover, the function $a\mapsto c_a$ is of class $C^{\min(d,r-1)}$, and the map $C\colon a\mapsto A_a(c_a,0)$ is of class $C^r$.  Hence it holds:
\begin{equation} \tag{$\mathcal T'$} \partial_x A_a(c_a,0)=0\; .\end{equation}

 %

The family of hyperbolic points $(\Omega_a)_a$ has a \emph{homoclinic $C^d$-paratangency} at $a=0$ (and $P$) if it satisfies $(\mathcal T)-(\mathcal T')$ and:
 \begin{equation} \tag{$\mathcal P$} 0=C(0)=\partial_a C(0)=\cdots = \partial_a^d C(0)\;.\end{equation}

The concept of homoclinic $C^d$-paratangency is dramatically  different to the one of \emph{non-degenerate homoclinic unfolding}: the non-degeneracy asks   $\partial_a C(0)\not=0$ instead of $(\mathcal P)$.

\section{Parablenders}\label{section:parablender}
\subsection{Blenders}
Blenders were first built in \cite{BD96} to construct $C^1$-robustly transitive diffeomorphism of $3$-manifolds which are homotopic to the identity.

A blender is a hyperbolic set with a strong stable direction. Projecting along the strong stable direction, we obtain a hyperbolic set for a surface local diffeomorphism that we shall called \emph{Blenders} as well.

Let us give a  paradigmatic example of a local diffeomorphism with a blender. Let:
\[f\colon (x,y)\in ([-1,-1/2]\sqcup [1/2,1])\times[-3/2,3/2]\mapsto \left\{\begin{array}{cl}
 (4 x+3, 2y/3-1/3)& x<0\\
  (4 x-3, 2y/3+1/3)& x>0\end{array}\right.\]   
	
	\begin{figure}[h]
    \centering
        \includegraphics[width=6cm]{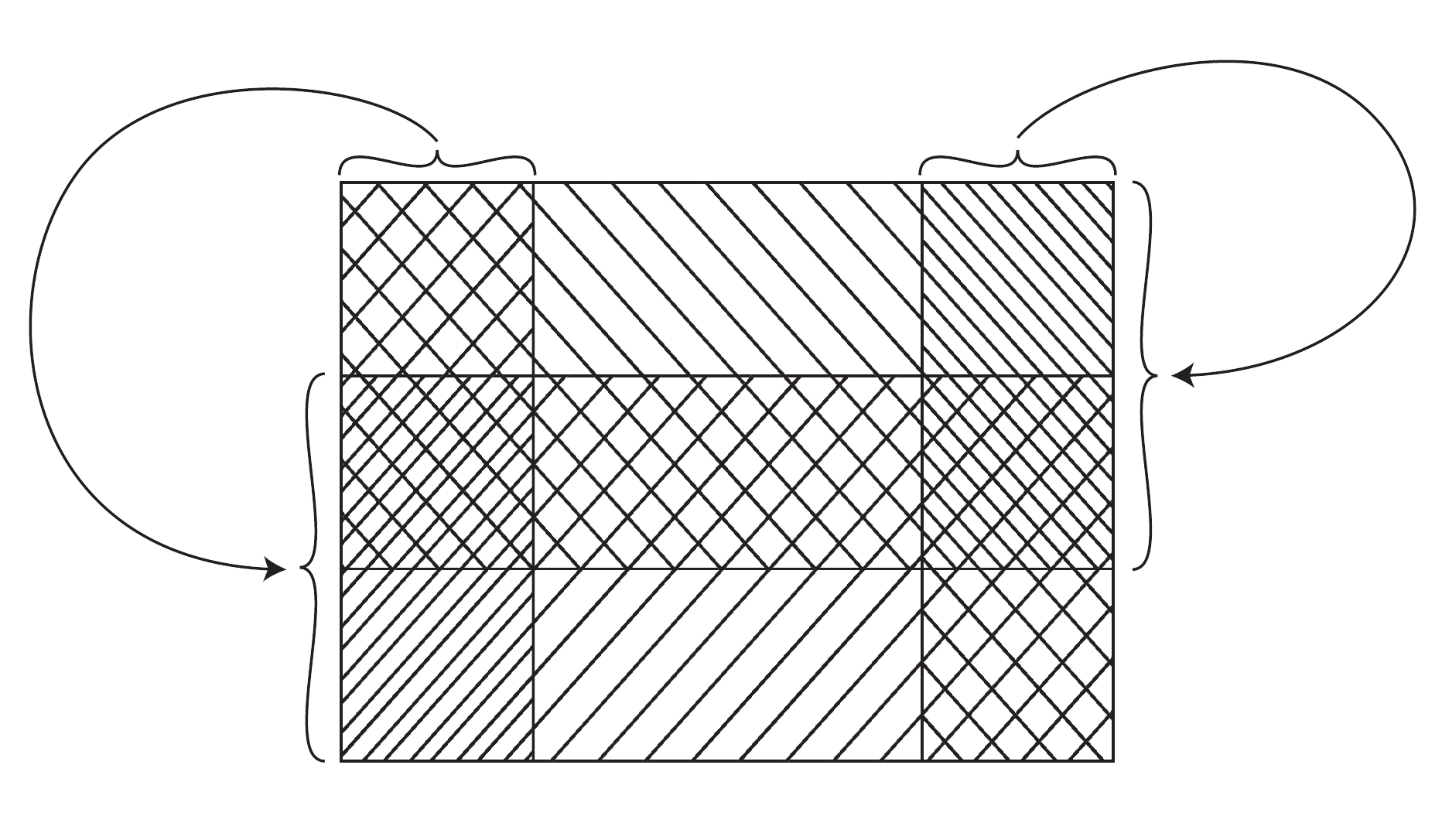}
    \caption{A paradigmatic example of a blender for a surface local diffeomorphism.}
    \label{preuvemain}
\end{figure}
	
The hyperbolic set $K:=\cap_{n\in \mathbb Z} f^n( ([-1,-1/2]\sqcup [1/2,1])\times[-3/2,3/2])$ is a \emph{blender}. We notice that the set 
$W^u_{loc}(K;f):= \cap_{n\ge 1} f^n( ([-1,-1/2]\sqcup [1/2,1])\times[-3/2,3/2])$ is equal to $[-1,1]^2$. It is actually the union of local stable manifolds $W^u_\loc(\underline k; f)$ which are of the form $[-1,1]\times \{y(\underline k)\}$, with $y(\underline k)\in [-1,1]$. For $f'$ $C^1$-close to $f$, let  $W^u_\loc(\underline k; f')$ be the curve close to $[-1,1]\times \{y(\underline k)\}$ given by hyperbolic continuation. 

By embedding $[-1,1]$ canonically into the circle $\mathbb S^1:= \mathbb R/6\mathbb Z$, the map $ f$ can be extended to a $4-1$-$C^\infty$-covering  of the annulus $\mathbb S^1 \times \R$ as follows:
\[f\colon (x,y)\in \mathbb R/6\Z \times \R\mapsto (4x -3, 2y/3+\rho(x)/3)\in \mathbb R/6\Z \times \R,\]
with $\rho\in C^\infty(\mathbb S^1, [-1,1]) $ equal to $1$ on $[1/2,1]$ and $-1$ on $[-1,-1/2]$.

Here is a fundamental property\footnote{The original fundamental property of blenders is (more) equivalent to the following:
 There exists a neighborhood $V_f$ of $f\in C^1(\R^2,\R^2)$ so that every segment of the form $\{0\}\times[y_1,y_2]$, with $-1/2<y_1<y_2<1/2$, intersects the unstable manifold of the hyperbolic continuation of the fixed point $(1,-1)$.} of this blender:
\begin{prop}\label{fondablender}
There exists a neighborhood $U_0$ of $f\in C^1(\R^2,\R^2)$ and a neighborhood $V_\Gamma$ of the curve $\Gamma=\{(x,2x^2):x\in [-1,1]\}$ in the $C^1$-topology, so that for every $f'\in V_f$, every $\Gamma'\in V_\Gamma$, there exists $\underline k\in \overleftarrow K$ so that $W^u_\loc(\underline k; f')$ is tangent to $\Gamma'$.\end{prop}

This proposition will be generalized (and proved) below, as the fundamental property of parablender  \ref{Ml2}.
\begin{rema} \label{BlenderNewhouse}
Given such a blender, it is easy to construct a Baire generic set of local-diffeomorphisms with infinitely many sinks (see \cite{DNP06} for more details). For this end, it suffices to suppose the existence of a dissipative hyperbolic fixed point $\Omega$ so that a segment of its stable manifold is in $V_\Gamma$ and so that its unstable manifold accumulates densely in the ones of $K$. Then by the fundamental property of blenders, for every map $g$ $C^r$-close to $f$, with $r\ge 1$, the stable manifold of $\Omega$ is tangent to a local unstable manifold of $K$ which is itself close to a segment of $W^u(\Omega;g)$. Thus one can find a small perturbation $g'$ of $g$ with a  homoclinic tangency. Then by using Proposition \ref{pointapuit}, one can create a new sink for a perturbation $g''$ of $g'$. Repeating this process inductively, one constructs a Baire generic set of local diffeomorphisms with infinitely many sinks.   
\end{rema}


\subsection{An example of Parablender} \label{defparablender}

Contrarily to blenders, the parablenders can uniquely be defined for families of dynamics $(f_a)_a$, say of class $C^d$. 
They are parameter family of hyperbolic set designed to produce not only homoclinic tangencies but also homoclinic $C^d$-paratangencies for perturbations of $f_a$ at every small parameter $a$. Indeed we will see in the next section that a $C^d$-paratangency implies the existence of a perturbation of the family having a new periodic sink which persists for a large open set of parameters (see Prop. \ref{sinkcreation}). Hence parablender will be crucial for the proof of Theorem \ref{main}. 

A blender satisfies that the tangent spaces of its unstable manifolds contain robustly a point of the tangent space of a parabola. 
Given a $C^d$-family of curves $(\gamma_a)_a$, we consider not only the tangent space $T\gamma_a$ of $\gamma_a$, but also the \emph{parameter $d$-Jet} $(\partial_a^j T\gamma_a)_{j=0}^d$ of $T\gamma_a$.
In contrast with the blender, a $C^d$-parablender  satisfies that the parameter $d$-Jets of the tangent spaces of its unstable manifolds
contain robustly a point of the parameter $d$-Jet of the tangent space of a parabola family. 

An animation drawing the dynamics of a blender and its fundamental property can be found in \cite{webpage}.  


\paragraph{Definition of $\mathring f$ and its Markov partition}
We are going to construct the parablender similarly. We shall first define the parablender of $k$-families of local diffeomorphisms, with $k=1$. Below we will give the modification for $k>1$. Also we now fix the integer $d\ge 0$.

Let $Q\colon x\in \mathbb S^1:= \mathbb R/6\mathbb Z\mapsto 4x-3\in \mathbb S^1$. We recall that $Q$ sends $[-1,-1/2]$ and  $[1/2,1]$ onto $[-1,1]$.
Let $\rho\in C^\infty (\mathbb S^1, [0,1])$ be a smooth function on $\mathbb S^1$ equal to $1$ on $[1/2,1]$ and $-1$ on $[-1,-1/2]$, with support in $[-3/2, -1/4]\cup [1/4, 3/2]$. 

We are going to look at a family  $(f'_a)_{a\in \R}$, with $f'_a$ close to the following map $\mathring f$:
\[ \mathring f\colon (x,y)\in \mathbb R/6\Z \times \R\mapsto (Q^{d+1}(x), 2y/3+\rho(x)/3)\in \mathbb R/6\Z \times \R\; .\]

\paragraph{Parablender Definition }
Let $\Delta_d:= \{-1,1\}^{d+1}$ be the set formed by $(d+1)$-upplets  $\delta= (\delta(i))_{0\le i\le d}\in \{-1,1\}^{d+1}$.

Let $I_0:=[-1,1]$,  $I_{-1}:= [-1,-1/2]$ and $I_{+1}:= [1/2,1]$. 

We associate to $\delta\in \Delta_d$ the following intervals:
\[I_{\delta} := (Q|I_{\delta(0)})^{-1} \circ \cdots  \circ (Q|I_{\delta(d)})^{-1}(I_0).\]

Put $Y_0:= [-1,1]\times [-3/2,3/2]$, $Y_{\delta}:= I_{\delta}\times [-3/2,3/2]$. We define:
$$\mathring f\colon (x,y)\in \sqcup_{\delta\in \Delta_d} Y_\delta=\left\{\begin{array}{cl}
(Q^{d+1}(x), 2y/3+ 1/3)& \delta(0)=1,\\
(Q^{d+1}(x), 2y/3- 1/3)& \delta(0)=-1.\end{array}\right.$$

We notice that $\mathring f$ sends $Y_{\delta}:= I_{\delta}\times [-3/2,3/2]$ onto 
$[-1,1]\times [-2/3,4/3]\subset Y_0$ if $\delta(0)=1$ and onto $[-1,1]\times [-4/3,2/3]\subset Y_0$ if $\delta(0)=-1$. 

Each $\delta\in \Delta_d$ is associated to a polynomial:
\[P_{\delta} (X)= \sum_{i=1}^d \frac{\delta(i)}{i!} X^i.\]

For every $\delta\in \Delta_a\cup\{0\}$, let $\tilde Y_{\delta}$ be a small compact neighborhood of $Y_\delta$, so that $(\tilde Y_\delta)_{\delta\in \Delta_a}$ are disjoint.

For every $f$ $C^r$-close to $\mathring f$, for $a\in (-1,1)$ and $\epsilon$ small ($\ll \max_{\delta} d(\tilde Y_\delta^c, Y_\delta)$), we define:
\[f_{\epsilon\, a}(z)=  f(z)+(0,\epsilon P_{\delta}(a)),\quad \text{if } z\in \tilde Y_{\delta}.\]

Hence, for $f=\mathring f$, on $Y_\delta$:
\[\left\{\begin{array}{cl}
f_{\epsilon\, a}(x,y)=(Q^{d+1}(x), 2y/3+1/3+\epsilon P_{\delta}(a))&\text{if }\delta(0)=1,\\
f_{\epsilon\, a}(x,y)=(Q^{d+1}(x), 2y/3-1/3+\epsilon P_{\delta}(a))&\text{if }\delta(0)=-1.\end{array}\right.\]

As the compact sets $(\tilde Y_\delta)_{\delta\in \Delta_d}$ are disjoint, we can extend the map $\sqcup_{\delta\in \Delta_a} Y_\delta\mapsto \epsilon P_{\delta}$ to a $C^r$-map of $\mathbb S^1\times \R$. Therefore:
\begin{fact} For every $f$ $C^r$-close to $\mathring f$, for every $\epsilon$ small, the map $f_{\epsilon\, a}$ can be extended to a $4^{d+1}$-$C^\infty$-covering of the annulus $\mathbb S^1\times \R$. Also the family $(f_{\epsilon\, a})_a$ is of class $C^{\infty,r}$.
\end{fact}

\begin{defi}[$C^d$-Parablender for $(f_a)_a$]
For every $f$  $C^r$-close to $\mathring f$ and $\epsilon$ small, let $\arr {B}_d(f_{\epsilon\, a})\subset (\tilde Y_0)^\Z$ be the subset of orbits of $f_{\epsilon\, a}$ which remain in $\cup_{\delta \in \Delta_d} \tilde Y_\delta $:
\[\arr {B}_d(f_{\epsilon\, a}) := \{(z_i)_{i\in \Z}\in (\cup_{\delta \in \Delta_d} \tilde Y_\delta)^\Z:\; z_{i+1} =f_{\epsilon\, a}(z_i)\}\; .\]\end{defi}
We notice that the compact set $\arr {B}_d(f_{\epsilon\, a})$ is hyperbolic and canonically homeomorphic to $ {\Delta}_d^\Z$. 
Let $B_d(f_{\epsilon\, a})$ be the image of $\arr {B}_d(f_{\epsilon\, a})$ by the zero-coordinate projection.

Any sequence $\underline \delta:= ( \delta_i)_{i}\in {\Delta}_d^{\Z}$ defines local stable and unstable manifolds:
\[W^s_{loc}(\underline  \delta, f_{\epsilon\, a}):= Y_0\cap \bigcap_{i\ge 0} f^{-i}_{\epsilon\, a}(\tilde Y_{ \delta_{i}})\quad \mathrm{and}\quad W^u_{loc}(\underline  \delta, f_{\epsilon\, a}):= Y_0\cap\bigcap_{i\ge 1} f^{i}_{\epsilon\, a}(\tilde Y_{ \delta_{-i}}).\]

Let $\underline \delta(a)$  be the unique intersection of $W^s_{loc}(\underline  \delta, f_{\epsilon\, a})$ and $W^u_{loc}(\underline  \delta, f_{\epsilon\, a})$.

When $f=\mathring f$, let us notice that $W^u_{loc}(\underline  \delta, \mathring f_{\epsilon\, a})$ is of the form 
$[-1,1]\times \{y(\underline  \delta,a)\}$, with $y(\underline  \delta,a)\in [-1,1]$, and the derivatives $(\partial^i_a  y(\underline  \delta,a))_{1\le i\le d}$  are small for $\epsilon$ small. Actually we can give the explicit expression of $y(\underline  \delta,a)$:
\[y(\underline  \delta,a)= \sum_{i=1}^\infty (2/3)^{i-1} \frac{\delta_{-i}(0)}{3} +\epsilon \sum_{i=1}^\infty (2/3)^{i-1}  P_{\delta_{-i}}(a).\] 
 
For any perturbation $f'_a$ of $f_a$, the hyperbolic set $\arr {B}_d(f_a)$ persists as a set $\arr {B}_d(f'_a)$, and its unstable manifolds as well. We denote by  $(W^u_{loc}(\underline  \delta, f'_a))_{\underline \delta}$ the unstable manifolds of $f'_a$ obtained by hyperbolic continuation.

A parablender satisfies a parameter version of the fundamental Property \ref{fondablender} of the blender. Robustly in the $C^{d,r}$-topology,  any $C^{d,r}$-family of parabolas close to the constant family $a\mapsto ((-1,1) \ni t \mapsto 2t^2\in \R)$ is $C^d$-paratangent to the parameter family of some of its unstable manifolds: 
\begin{propfonda}\label{Ml2}
For every $\infty \ge r>d\ge 0$ or $\infty > r = d\ge 2$, 
there exist $\alpha>0$, a neighborhood $U_0$ of $\mathring f$ and a neighborhood $V_\gamma$ of the constant family $a\mapsto ((-1,1) \ni t \mapsto 2t^2\in \R)$ in $C^d((-1,1), C^r([-1,1],\R))$, such that for every $\epsilon>0$ small, there exists a neighborhood $U_\epsilon$ of $\{ (f_{a\, \epsilon})_a\colon f\in U_0\}$ in the $C^{d,r}$-topology satisfying the following property: 

For every $(f_a)_a\in U_\epsilon$, every $(\gamma_a)_a\in V_\gamma$ and  every $|a_0|<\alpha$, there exists  $\underline  \delta\in {\Delta}_d^{\Z^-}$ such that $W^u_{loc}(\underline \delta;f_{a_0})$ is tangent to $Graph\, \gamma_{a_0}$ and the families $(W^u_{loc}(\underline  \delta;f_{a}))_a$ and $(Graph\, \gamma_a)_a$ are $C^{d}$-paratangent at $a_0$.
\end{propfonda}
\begin{rema} By Proposition \ref{cartecool4cpct}, the family  $(W^u_{loc}(\underline  \delta;f'_{a}))_a$ is of class $C^{d,r}$, so the $C^d$-paratangency makes sense in this fundamental property.
\end{rema}
This fundamental property is shown in \textsection \ref{proofMl2} below.

\paragraph{Parablender for the case $k>1$.}
Let $r>d\ge 0$ and $k\ge 1$.

Parablender are designed to get robust $C^d$-paratangencies. For one-dimensional parameter families, we defined the parablender thanks to a blender with $2^{d+1}$-symbols, since Equality $(\mathcal P)$ involves $d+1$ derivatives.

 When the families are $k$-dimensional, Equality $(\mathcal P)$ involves $d'$-derivatives, 
with $d'$ be the dimension of the space $\{P\in P[X_1,\dots, X_k]:\; deg \; P\le d,\; P(0)=0\}$. Hence we shall define the parablender thanks to $2^{d'}$-symbols.

Let $(Q_i)_{i=1}^{d'}$ be a canonical base of $\{P\in P[X_1,\dots, X_k]:\; deg \; P\le d,\; P(0)=0\}$. Hence each $Q_i$ is of the form $X_1^{i_1}\cdots X_k^{i_k}$ with $1\le i_1+\dots +i_k\le d$. For $\delta\in \Delta_{d'}$ we put:
\[P_\delta(X_1,\cdots ,X_k)= \sum_{i=1}^{d'} \frac{\delta(i)}{i_1!\cdots i_k!} X_1^{i_1}\cdots X_k^{i_k}.\]
For the case $k>1$, we shall consider instead:
\[ \mathring f\colon (x,y)\in \mathbb R/6\Z \times \R\mapsto (Q^{d'+1}(x), 2y/3+\rho(x)/3)\in \mathbb R/6\Z \times \R,\]
And for a map $f$ $C^r$-close to $\mathring f$, $a=(a_1,\dots ,a_k)\in (-1,1)^k$, we define for every $\delta\in \Delta_{d'}$:
 \[f_{\epsilon\, a}(z)=  \mathring f(z)+(0,\epsilon P_{\delta}(a_1,\dots ,a_k)),\quad \text{if } z\in \tilde Y_{\delta}.\]
Likewise it is easy to extend $f_{\epsilon\, a}$ from $\cup_{\delta \in \Delta_{d'}} \tilde Y_{\delta}$ to $\mathbb S^1\times \R$ as a $C^r$-map, so that $(f_{\epsilon,\, a})_a$ is of class $C^{\infty, r}$.

\begin{defi}[$C^d$-Parablender for $(f_a)_a$]
For every $f$  $C^r$-close to $\mathring f$ and $\epsilon$ small, let $\arr {B}_{d\, k}(f_{\epsilon\, a})\subset (\tilde Y_0)^\Z$ be the subset of orbits of $f_{\epsilon\, a}$ which remain in $\cup_{\delta \in \Delta_{d'}} \tilde Y_\delta $:
\[\arr {B}_{d\, k}(f_{\epsilon\, a}) := \{(z_i)_{i\in \Z}\in (\cup_{\delta \in \Delta_{d'}} \tilde Y_\delta)^\Z:\; z_{i+1} =f_{\epsilon\, a}(z_i)\}\; .\]\end{defi}
We notice that the compact set $\arr {B}_{d\, k}(f_{\epsilon\, a})$ is hyperbolic and canonically homeomorphic to $ {\Delta}_{d'}^\Z$. 
Let $B_{d\, k}(f_{\epsilon\, a})$ be the image of $\arr {B}_{d\, k}(f_{\epsilon\, a})$ by the zero-coordinate projection.

Likewise, any sequence $\underline \delta:= ( \delta_i)_{i}\in {\Delta}_{d'}^{\Z}$ defines local stable and unstable manifolds:
\[W^s_{loc}(\underline  \delta, f_{\epsilon\, a}):= Y_0\cap \bigcap_{i\ge 0} f^{-i}_{\epsilon\, a}(\tilde Y_{ \delta_{i}})\quad \mathrm{and}\quad W^u_{loc}(\underline  \delta, f_{\epsilon\, a}):= Y_0\cap\bigcap_{i\ge 1} f^{i}_{\epsilon\, a}(\tilde Y_{ \delta_{-i}}).\]
 
For any perturbation $f'_a$ of $f_a$, the hyperbolic set $\arr {B}_{d\, k}(f_a)$ persists as a set $\arr {B}_{d\, k}(f'_a)$, and its unstable manifolds as well. We denote by  $(W^u_{loc}(\underline  \delta, f'_a))_{\underline \delta}$ the unstable manifolds of $f'_a$ obtained by hyperbolic continuation.

For $k$-parameter families, the set of its unstable manifolds satisfies the same fundamental property:
\begin{propfonda}\label{Ml2k}
For every $\infty \ge r>d\ge 0$ or $\infty >r= d\ge 2$, there exist $\alpha>0$, a neighborhood $U_0$ of $\mathring f$ and a neighborhood $V_\gamma$ of the constant family $a\mapsto ((-1,1) \ni t \mapsto 2t^2\in \R)$ in $C^d((-1,1)^k, C^r([-1,1],\R))$, such that for every $\epsilon>0$ small, there exists a neighborhood $U_\epsilon$ of $\{ (f_{a\, \epsilon})_a\colon f\in U_0\}$ in the $C^{d,r}$-topology satisfying the following property: 

For every $(f_a)_a\in U_\epsilon$, every $(\gamma_a)_a\in V_\gamma$ and every $\|a_0\|<\alpha$, there exists  $\underline  \delta\in {\Delta}_{d'}^{\Z^-}$ such that $W^u_{loc}(\underline \delta;f_{a_0})$ is tangent to $Graph\, \gamma_{a_0}$ and the families $(W^u_{loc}(\underline  \delta;f_{a}))_a$ and $(Graph\, \gamma_a)_a$ are
$C^{d}$-paratangent at $a_0$.
\end{propfonda}
The proof of this fundamental property is similar to the one of Property \ref{Ml2}. It is left to the reader.
\subsection{Proof of fundamental property of the parablenders \ref{Ml2}} \label{proofMl2}

Given a  families of dynamics $(f_a)_a$, we are going to define a sequence of preimages $((\gamma_a^{n})_a)_{n\le 0}$ of the family of parabola $(\gamma_a)_a$.
We proceed by induction. At the step $n=0$, clearly $(\gamma_a^{0})_a$ is equal to 
$(\gamma_a)_a$. At the step $n-1<0$, we will suppose that the derivatives w.r.t. $a$ of the $y$-coordinate of the minimum of $\gamma_a^{n}$ are small, and we will chose the inverse branch 
of $(f_a)_a$ so that the derivatives of $y$-coordinate of the minimum of $\gamma_a^{n-1}$ remain small.
The choice of these inverse branches will define the unstable manifold of the parablender to which $(\gamma_a)_a$ is $C^d$-paratangent.

First we are going to study the property of the  $y$-coordinate of the minimum of parabolas family. Then we will define these sequences of inverse branches when $(f_a)_a$ is of the form  $(f_{a\, \epsilon})_a$ for $f$ $C^r$-close to $\mathring f$, to exhibit the existence of robust $C^d$-paratangency. Finally, we will show the general case.

\subsubsection{Parabolas and function $\min\gamma$}
\begin{defi}[Function $\min\gamma$]\label{mingamma}
For every continuous family $(\gamma_a)_{a}$ of functions $\gamma_a\in C^0(I ,\R)$, with $I$ an interval of $\R$, let us define:
\[a\mapsto \min\gamma(a):= \inf_{t\in I} \gamma_a(t)\in \R.\]
\end{defi}

\begin{defi}[$C^r$-Parabola]
 For $r\ge 1$, a \emph{$C^r$-parabola} is the graph of a function $\gamma\in C^r(I, \R)$ over a  segment $I$ of $\mathbb R$ so that:
\begin{itemize}
\item At the end point of $I$ the value of $\gamma$ is $\ge 3/2$: 
\[\gamma|\partial I \ge 3/2,\]
\item the minimum of $\gamma$ satisfies:
\[|\min \gamma|\le 2/3\]
\item If $r\ge 2$, it holds $\gamma''(t)\ge 1$ for every $t\in I$.
\end{itemize}
\end{defi}

\begin{lemm}\label{scholemme}
If the function  $(t,a)\mapsto \gamma_a(t)$ is continuous then $\min\gamma$ is continuous with small uniform norm.

If $(\gamma_a)_a$ is a $C^d$-family of parabolas, with $r>d\ge 0$ or $r= d\ge 2$, then the function $a\mapsto \min \; \gamma(a)$ is of class $C^d$.

\end{lemm}
\begin{proof} The case $d=0$ is clear. The case $\infty \ge r\ge d\ge 1$ with  $r\ge 2$ follows from  Fact \ref{prametrizationtangent}.\end{proof}

\begin{rema}
On the other hand, for $d\ge 1$ and $r=1$, the function $\min\gamma$ is not always of class $C^1$: indeed, one can construct a family $(\gamma_a)_a$ with two local minima having the same value at $a=0$, one of which is increasing with $a$, the other being decreasing. The function $\min \gamma$ is then not differentiable at $a=0$. This is why  we suppose $d< r$ in the fundamental property of parablender.
\end{rema}
\subsubsection{Verification of the fundamental property for the family $(\mathring f_{\epsilon \, a})_a$ at $a_0=0$}
Let $r>d\ge 0$ and let  $(\gamma_a)_{a\in(-1,1)}$ be a family $C^{d,r}$-close to the constant family $(t\in [-1,1]\mapsto 2t^2)_a$. 
Let us show that there exists $\underline \delta\in {\Delta}_d^{\Z}$ such that $W^u_{loc}(\underline \delta;\mathring f_{\epsilon \, 0})$ is tangent to $Graph\, \gamma_{0}$ and $(W^u_{loc}(\underline \delta; \mathring f_{\epsilon \, 0}))_a$ and $(Graph\, \gamma_a)_a$ are $C^d$-paratangent at $a=0$.

We are going to define by decreasing induction on $n\le 0$, a symbol $\delta_{-n-1}\in \Delta_d$  and a $C^{d,r}$-family of parabolas $(\gamma^n_a)_a$ defined for $t$ in an interval of $[-1,1]$ and $a$ close to $0$, so that:
\begin{itemize}
\item $\gamma_a^0:= \gamma_a$,
\item the graph of  $\gamma_a^n$ is a $C^r$-parabola,  
\item the function $a\mapsto \min\gamma^n(a)\in \R$ defined in \ref{mingamma} satisfies:
\begin{equation}\tag{$\dagger$} |\min \gamma^n(0)|<2/3,\quad |{\partial_a^i \min\gamma^n(0)}|\le 2\epsilon ,\quad \forall 1\le i\le d.\end{equation}
\item  $Graph\, \gamma^{n+1}_a= \mathring  f_a\left(Y_\delta\cap Graph\, \gamma^{n}_a\right)$.
\end{itemize}
After proving the induction, we will see that with $\underline \delta:= (\delta_i)_i$ (with any $\delta_i\in \Delta_d$, for $i\ge 0$) satisfies that 
$W^u_{loc}(\underline \delta;\mathring f_{\epsilon \, a})$ and $(Graph\, \gamma_a)_a$ are $C^d$-paratangent at $a=0$.
For the step $n=0$, it suffices to take $V_\gamma$ small enough so that every $\gamma_a^0:= \gamma\in V_\gamma$ has a graph which is a $C^r$-parabola and satisfies  ($\dagger$).

Let $\delta_{-1}(i) $ be the sign of $\partial_a^i \min \gamma^0(0)$ for every $i\le d$ and let $\delta_{-1}= (\delta_{-1}(i))_{i=0}^d\in \Delta_d$.

Let $\gamma_a^{-1}$ be the preimage of $\gamma^0_a$ by $\mathring f_a|\tilde Y_{\delta_{-1}}$. The function $\min\gamma^{-1}(a)$, equal to the smallest $y$-coordinate of $\gamma^{-1}_a$, has its $d$-first derivatives which can be given explicitly:
\[\min \gamma^{-1}(0) = \frac32 (\min\gamma^0(0)-  \delta_{-1}(0)/3),\quad 
\partial_a^i \min\gamma^{-1}(0)= \frac32 ( \partial_a^i \min\gamma^0(0)-\epsilon \delta_{-1}(i)), \quad \forall 1\le i\le d.\]

We notice that $|\min\gamma^0(0)-  \delta_{-1}(0)/3|< 1/3$ and $|\partial_a^i \min\gamma^0(0)-\epsilon \delta^0(j)|< \epsilon $, for all $d\ge i\ge1$, hence 
\[|\min\gamma^{-1}(0)|\le  \frac32 \frac13  <2/3\quad \mathrm{and}\quad  |\partial_a^i \min\gamma^{-1}(0)|\le  \frac32 \epsilon <2\epsilon.\]

Furthermore, it is easy to see that the values of $\gamma^{-1}$ at its end points of its definition domain is greater than 
$\min(3/2 (\gamma^0|\{-1,1\}), 3/2)$. As $\gamma^0|\{-1,1\}$ is close to $3/2$, the values of $\gamma^{-1}$ at its end points of its definition domain is $3/2$.

When $r\ge 2$, it is also easy to see that $\partial_t^2 \gamma^{-1}$ is greater than $\partial_t^2 \gamma^{0}\ge 1$. Thus $Graph\; \gamma^{-1}$ is indeed a parabola.

Let $n\le -1$, and suppose $\gamma^n_a$ constructed for $a$ small so that $\gamma^n$ has a graph which is a parabola and satisfies $(\dagger)$. Let 
$\delta_n(i)\in\{-1,1\} $ be the sign of $\partial_a^i \min\gamma^n(a)$ for every $i\le d$ and put $ \delta_n= (\delta_n(i))_{i=0}^d\in \Delta_d$.
Let $\gamma_a^{n-1} $ be the preimage of $\gamma^n_a$ by $\mathring f_{\epsilon\, a}|\tilde Y_{m_n}$. As above, we show that $\min\gamma^{n-1}$ satisfies $(\dagger)$ and its graph is a parabola. This finishes the proof of the induction hypothesis.

Let us show that the local unstable manifold $W^u_{loc}(\underline \delta;\mathring f_{\epsilon\, 0})$
is tangent to $\gamma_{0}$, and that $(W^u_{loc}(\underline \delta;\mathring f_{\epsilon\, a}))_a$ is  $C^d$-paratangent to $(\gamma_a)_a$ at $a=0$.

For this end it suffices to remark that for every small $a$:
\[\min\gamma(a)= (2/3)^N \min \gamma^{-N}(a)+ \sum_{n=1}^N (2/3)^{n} (\delta_{-n}(0)/3+ \epsilon P_{\delta_{-n}}(a)),\quad \forall N\ge 1,\]
and that:
\[W^u_{loc}(\underline \delta;\mathring f_{\epsilon \; a})= [-1,1]\times\{ y_{\underline \delta,a}\}\quad \text{with}\; y_{\underline \delta,a}:=\sum_{n=1}^\infty (2/3)^{n} (\delta_{-n}(0)/3+ \epsilon P_{\delta_{-n}}(a)).\]

We see that $\min\gamma(0)$ converges to $ y_{\underline \delta,0}$. Thus there exists a tangency and so a $C^0$-paratangency. This solves the case $d=0$ and so $r=1$.

When $r\ge 2$, we put $\tau(a):= (0,-\eta(a))$ with
$$\eta(a)= \min\gamma(a)-y_{\underline \delta,a}= (2/3)^N \min\gamma^N(a) -\sum_{n=N+1}^\infty (2/3)^{n} (\delta_{-n}(0)/3+ \epsilon P_{\delta_{-n}}(a)).$$

We remark that $W^u_{loc} (\underline \delta; \mathring f_{\epsilon\, a})+\tau(a)$ is tangent to $Graph\, \gamma_a$ for $a$ small.
Also the map $\eta$ is a function of class $C^d$ of $a$. By $(\dagger)$, the $d$-first derivatives of $(2/3)^N \min \gamma^{-N}(a)$ at $a=0$ are small when $N$ is large. Furthermore,   the $d$-first derivatives of $\sum_{n=N+1}^\infty (2/3)^{n} (\delta_{-n}(0)/3+ \epsilon P_{\delta_{-n}}(a))$ are small as well. Consequently, the $d$-first derivatives of $\eta$ at $a=0$ are small when $N$ is large and so equal to $0$. This proves the $C^d$-paratangency.
\label{section3.2.3}
\subsubsection{Proof of the fundamental property in its whole generality}

Let $\infty \ge r>d\ge 0$ or $\infty >r = d\ge 2$. Let us show the existence of:
\begin{itemize}
\item a $C^r$-neighborhood $U_0$ of $\mathring f$, 
\item  a $C^{d,r}$-neighborhood $U_\epsilon$ of $\{(f_{\epsilon\, a})_a:\; f\in U_0\}$ for every $\epsilon>0$ small,
\item  a neighborhood $V_\gamma$ of $(t\mapsto 2t^2)_a$,
\end{itemize}
so that for every $(f_a)_a\in U_\epsilon$, $(\gamma_a)_a\in V_\gamma$ and $a_0\in V_a$, there exists $\underline \delta\in {\Delta}_d^{\Z}$ satisfying that $W^u_{loc}(\underline \delta;f_{0})$ is tangent to $Graph\, \gamma_{0}$ and the family $(W^u_{loc}(\underline \delta;f_{a}))_a$ is $C^d$-paratangent to $(Graph\, \gamma_a)_a$ at $a=0$.

\paragraph{Generalization to the case $a_0 \not=0$} We notice that for every $C^{d,r}$-families $(f_a)_a$ and $(\gamma_a)_a$, 
for $a_0$ small, the family $(f_{a_0+a})_a$ and $(\gamma_{a+a_0})_a$ are $C^{d,r}$-close to $(f_a)_a$. Thus, by shrinking $V_\gamma$ and $U_\epsilon$, the above statement implies the existence of:
\begin{itemize}
\item a $C^r$-neighborhood $U_0$ of $\mathring f$, 
\item  a $C^{d,r}$-neighborhood $U_\epsilon$ of $\{(f_{\epsilon\, a})_a:\; f\in U_0\}$ for every $\epsilon>0$ small,
\item a neighborhood of $V_a$ of $0\in [-1,1]$,
\item  a neighborhood $V_\gamma$ of $(t\mapsto 2t^2)_a$,
\end{itemize}
so that for every $(f_a)_a\in U_\epsilon$, $(\gamma_a)_a\in V_\gamma$ and $a_0\in V_a$, there exists $\underline \delta\in {\Delta}_d^{\Z}$ satisfying that $W^u_{loc}(\underline \delta;f_{a_0})$ is tangent to $Graph\, \gamma_{a_0}$ and the family $(W^u_{loc}(\underline \delta;f_{a}))_a$ is $C^d$-paratangent to $(Graph\, \gamma_a)_a$ at $a=a_0$.

Consequently we can study the case $a_0=0$.

Let  $H_\epsilon$ be the $\epsilon^2$-neighborhood of $(id_{\R^2})_a$ in $C^d((-1,1)^k, Diff^r(\R^2))$. In what follows $\epsilon$ is supposed small.

Let us now define $V_\gamma$, $U_0$ and $U_\epsilon$.

\paragraph{Definition of $V_\gamma$.} Let $V_\gamma$ be a small $C^{d,r}$-neighborhood of $(t\mapsto 2t^2)_a$ so that for all $(\gamma_a)_{a\in(-1,1)}\in V_\gamma$, for every $a_0\in (-1,1)$, the $Graph\, \gamma_{a_0}$ is a $C^r$-parabola and it holds:
\[ |{\partial_a^i \min\gamma^0(a_0)}|\le \epsilon,\quad \forall 0\le i\le d.\]  
We notice that for every $h=(h_a)_a\in H_\epsilon$, with $h^*\gamma_a^n$ the function whose graph is equal to the image by $h_a$ of $Graph\, \gamma_a^n$, the family $(h^*\gamma_a^n)_a$  satisfies:  
\[ |{\partial_a^i \min h^*\gamma^0(a_0)}|\le 2\epsilon,\quad \forall 0\le i\le d.\] 


\paragraph{Definition of $U_0$ and $U_\epsilon$ for $\epsilon>0$ small enough.} 

\noindent $(a)$ In section \ref{section3.2.3}, we saw that  for every $a$, the local unstable manifold $W^u_{loc} (\underline \delta; \mathring f_{\epsilon\, a})$ is of the form $[-1,1]\times \{y(a)\}$. Hence, by Proposition \ref{cartecool4cpct}, for $U_0$ and then $U_\epsilon$ small enough, for every $(f_a)_a\in U_\epsilon$ for every $\underline \delta\in \Delta_d^\Z$, there exists $(h_a)_a\in H_\epsilon$ so that  $h_a(W^u_{loc} (\underline \delta; f_a))$ is of the form $[-1,1]\times \{y(a)\}$ for every $a$.

Furthermore,  for every  $f'_a=f_a\circ h_a^{-1}$ with $(f_a)_a\in U_\epsilon$ and $(h_a)_a\in H_\epsilon$, we ask that for every vector $v=(v_x,v_y)$ such that $|v_x|\le |v_y|$, for every $\delta\in \Delta_d$, for very $z\in \tilde Y_\delta$, the vector $v'=(v'_a,v'_y):= (D_zf_a')^{-1} (v)$ satisfies 
\begin{equation}\label{cone}\tag{$\mathcal C$} |v'_x|\le |v'_y|\quad \text{and}\quad |v'_y|\ge (4/3) |v_y|.\end{equation}

\noindent $(b)$ We recall that $\tilde Y_{\delta}$ is a small neighborhood of $ Y_{\delta}$ for every $\delta\in \Delta_d$. Let us suppose that
$\tilde Y_{\delta}$ is of the form $\tilde Y_{\delta}= \tilde I_\delta\times(-3/2-\mu ,3/2+\mu)$,  with $\tilde I_\delta$ the $\mu$-neighborhood of $I_\delta$, for $\mu$ small. 

We can suppose $\mu$ small enough so that for every $\delta\in \Delta_d$, the image of $\tilde Y_{\delta}$ by $\mathring f$ contains
 $[-1-4\mu,1+4\mu]\times [-1-2\mu/3, 2/3 +2\mu/3]$ if $\delta(0)=-1$ and $[-1-4\mu,1+4\mu]\times [-2/3-2\mu/3, 1 +2\mu/3]$ if $\delta(0)=1$.

 Then for $U_0$ small enough,  for every $f\in U_0$, the image of $\tilde Y_{\delta}$ by $ f$ contains $[-1-3\mu,1+3\mu]\times [-2/3-\mu/2, 2/3+\mu/2]$.
 For $\epsilon$ small enough and the neighborhood $U_\epsilon$ of $\{(f_{\epsilon\, a})_a:\; f\in U_0\}$ small enough, for every $(f_{a})_a\in U_\epsilon$, every $a\in (-1,1)$, the image of $\tilde Y_{\delta}$ by $ f_{a}$ contains $[-1-2\mu,1+2\mu]\times [-2/3, 2/3]$.
This condition will be important to prove in $(c)$ that  a parabola
which is the graph of a function $\gamma$ on a segment $I\subset [-1-\mu, 1+\mu]$, has its preimage by $f_a|\tilde Y_\delta$ which is a parabola which is the graph of a function $\gamma'$ on a segment $I'\subset [-1-\mu, 1+\mu]$.   

\noindent $(c)$  Let us ask for a third condition on $U_0$ and $U_\epsilon$.

Let us consider any  $C^d$-family $(\gamma_a)_a$ of $C^r$-function $\gamma_a$ on a segment $I_a\subset [-1-\mu, 1+\mu]$, defined for $a$ small, so that:
\[\gamma(0)|\partial I_a\ge 3/2,\quad |\min \gamma(0)|\le 2/3,\quad  |\partial_a^i \min \gamma(0)|\le 2\epsilon \quad \forall 1\le i\le d,\]
and such that if $r\ge 2$, then $\partial_t^2 \gamma \ge 1$. In particular, $Graph\, \gamma_a$ is a parabola.

Let $ \delta=(\delta(i))_i\in \Delta_d$ be so that $\delta(i)$ is of the sign of  $\partial_a^i \min \gamma (0)$. Then  
\[\gamma(0)|\partial I_a-\delta(0)/3\ge 7/6,\quad |\min \gamma(0)-\delta(0)/3|\le 1/3,\]
\[ |\partial_a^i \min \gamma(0) -\epsilon \delta(i)|\le \epsilon\quad \forall 1\le i\le d.\]
As in the restricted case of section \ref{section3.2.3}, for $a$ close to $0$, 
the curve $\mathring f_{a\, \epsilon}^{-1}|\tilde Y_{\delta}(Graph\; \gamma_a)$ is also the graph of a parabola $\gamma_a^{-1}$ over a segment $J_a$ of $[-1-\mu, 1+\mu]$, which satisfies:
 \[\gamma^{-1}(a)|\partial J_a \ge \min(3/2+\mu, \frac{7}6 \frac32)=3/2+\mu,\quad 
|\min \gamma^{-1}({a})|\le \frac{3/3}{2}<2/3,\]
\[|\partial_a^i \min \gamma^{-1}({a})|\le \frac{3\epsilon }{2}< 2\epsilon\;. \]
By $(b)$, we can chose $U_0$ and then $U_\epsilon$ small enough so that for every such a family of parabolas  $(\gamma_a)_a$, for every 
$(f_a)_a\in U_\epsilon$, the curve $f_a^{-1}|\tilde Y_{\delta}(Graph\; \gamma_a)$  is also the graph of a parabola $\gamma'^{-1}_a$ over a segment  $J'_a$ of $[-1-\mu, 1+\mu]$, and such that it satisfies for $a$ small:
\begin{equation} \tag{$\star$ }\gamma'^{-1}_a|\partial J'_a\ge 3/2,\quad 
|\min \gamma'^{-1}({0})|\le 2/3 \quad |\partial_a^i \min \gamma'^{-1}(0)|\le 2\epsilon\;.\end{equation}

Moreover, we can suppose that the latter property $(\star)$ holds as well for every $(f'_a)_a$ of the form $f'_a = f_a\circ h^{-1}_a$, where $f_a\in U_\epsilon$ and $(h_a)_a\in H_\epsilon$.

\paragraph{Induction hypothesis}


Let $(\gamma_a)_a\in V_\gamma$, $(f_a)_a\in U_\epsilon$, and $a\in V_a$. As in the above restricted case, we define by decreasing induction on $n\le 0$, a symbol $\delta_n\in \Delta_d$ and a $C^d$-family of parabolas $(\gamma^n_a)_a$ defined for $t$ on a segment $I_a^n$ of $[-1-\mu, 1+\mu]$ for $a$ small, so that:
\begin{itemize}
\item[$(o)$]  $\gamma^0_a=\gamma_a$. 
\item[$(i)$] For every $(h_a)_a\in H_\epsilon$, the curve $h_a(Graph\, \gamma^n_a)$ is a $C^r$-parabola. 
\item[$(ii)$]  For every $h=(h_a)_a\in H_\epsilon$, with $h^*\gamma_a^n$ the function whose graph is equal to the image by $h_a$ of $Graph\, \gamma_a^n$, the family $(h^*\gamma_a^n)_a$  satisfies: 
\begin{equation}\tag{$\dagger'$} |{ \min\, (h^*\gamma)(0)}|\le 2/3,\quad|{\partial_a^i  \min\, (h^*\gamma)(0)}|\le 2\epsilon,\quad \forall 1\le i\le d.\end{equation}
\item[$(iii)$]  $Graph\, \gamma^{n+1}_a= f_a\left(Y_{m_{-n}}\cap Graph\, \gamma^{n}_a\right)$ and $\gamma^{0}_a=\gamma_a$ for every $a$ small.
\end{itemize}

Let us begin with the step $n=0$. By definition of $V_\gamma$, the function $\gamma^{0}_a=\gamma_a$ satisfies induction hypothesis $(o)$, $(i)$ and $(ii)$ ($(iii)$ does not make sense at this step). 

Let $n\le 0$ and suppose $(\gamma^n_a)_a$ constructed and satisfying $(i)$. Let $\delta_{n-1}(i) $ be the sign of $\partial_a^i \min \gamma^n(0)$ for every $i\le d$ and let $\delta_{n-1}= (\delta_{n-1}(i))_{i=0}^d\in \Delta_d$.

Let $\gamma_a^{n-1}$ be the preimage of $\gamma^n_a$ by $f_a|\tilde Y_{m_{-1}}$. By definition of $U_\epsilon$ and $(c)$, the curve $\gamma^{n-1}_a$ is a parabola, and the family 
$(\gamma^{n-1}_a)_a$ satisfies induction hypothesis $(ii)$.

This finishes the proof of the induction.
\paragraph{Proof of the $C^d$-paratangency}
Chose any $\delta_i\in \Delta_d$ for $i\ge 0$, and put $\underline \delta=(\delta_i)_i$.  Let us show that the local unstable manifold defined by 
\[W^u_{loc}(\underline \delta;f_{a}):= Y_{0}\cap \bigcap_{n\ge 1} f_a^n(\tilde Y_{\delta_{-n}})\]
satisfies that $(W^u(\underline \delta;f_{a}))_a $ is $C^d$-paratangent to $(Graph\, \gamma_a)_a$ at $a=0$.

Let $\sigma\colon \Delta_d^\mathbb Z\to \Delta_d^\Z$ be the shift.

By $(a)$, for every $j\ge 0$, there exists $h^j=(h^j_a)_a\in H_\epsilon$ so that 
$h^j_a(W^u_{loc} (\sigma^{-j}(\underline \delta); f_a))$ is horizontal, that is of the form $[-1,1]\times \{y_j(a)\}$ with $a\mapsto y_j(a)$ of class $C^d$.

By Remark  \ref{Remprametrizationtangent}, the concept of paratangency is invariant by such a coordinates change. Hence it suffices to prove that $(h^0_a W^u_{loc} (\sigma^{-j}(\underline \delta); f_a))_a$ is $C^d$-paratangent to $(h^0_a(Graph \, \gamma_a))_a=(Graph \, h^{0*}\gamma_a)_a$ at $a=0$.

We observe that there exists $\tau_0=(0,\eta_0)$ such that $h^0_{0} W^u_{loc} (\underline \delta; f_{0})+\tau_0$ is tangent to $h^0_0(Graph\, \gamma_{0})$. The tangency is quadratic if $r\ge 2$.

By Fact \ref{prametrizationtangent}, there exists a unique function $a\mapsto \eta(a)$ so that $\eta(0)=\eta_0$ and for every $a$ small, with $\tau(a)= (0,\eta(a))$, the curve $h^0_a W^u_{loc} (\underline \delta; f_{a})+\tau(a)$ is tangent to $h^0_a(Graph\, \gamma_{a})$. Let $Q_a=(q_a, h^{0*}_a\gamma_a(q_a))$ be the tangency point.

For every $a$ small, let $C_a$ be the vertical curve $\{(q_a, h^{0*}_a\gamma_a(q_a)-t)\colon t\in [0, \eta(a)]\}$. The curve $C_a$ has length $|\eta_a|$. 
We notice that there exists $N_a\in \mathbb N$ large when $a$ is small such that $C_a\subset h^{0}_a(Y_0\cap \bigcap_{i=1}^{N_a} f^i_a(\tilde Y_{\delta_{-i}}))$.

For $n\le N_a$, Let $C_a^n:= h^n_a(f^{-n}_a((h^{0}_a)^{-1}( C_a))\cap \bigcap_{i=1}^{n} f^{i-n}_a(\tilde Y_{\delta_{-i}}))$. It is a connected curve. By $(a)$ inequality (\ref{cone}), the tangent space of $C_a^n$ has non zero vertical component at every point. In other words, with $p_y\colon \R^2\to \R$ the second coordinate projection, at every $z\in C_a^n$, we have $p_y(T_z C_a^n)\not =0$.
Moreover, the length of $p_y(C_a^n)$ is at least $(4/3)^n$ times the length of $C_a$:
\[ |p_y(C_a^n)|\ge (4/3)^n |p_y(C_a^n)|= (4/3)^n |C_a| =(4/3)^n |\eta(a)|.\]
On the other hand, $|p_y(C_a^n)|$ is smaller than $4/3$ by $(ii)$, inequality $(\dagger')$. Thus for $a=0$, 
we have $|\eta(0)|\le (3/4)^n 4/3$.
Taking $n\to \infty$, it comes $\eta(0)=\eta_0=0$. In particular,
$W^u_{loc} (\underline \delta; f_{0})$ is tangent to $Graph\, \gamma_{0}$.

By induction on $0\le i\le d-1$, let us assume that $\partial_a^j \eta(0)=0$ for every $j\le i$.
This means that there is a $C^i$-paratangency between $(h_a^0W^u(\underline \delta; f_a))_a$ and $(h_a^0Graph\, \gamma_a)_a$. Hence, by invariance of the $C^i$-paratangency by coordinates change, for every $n\ge 0$,   there exists a $C^i$-paratangency between 
 $(h_a^nW^u(\sigma^{-n}(\underline \delta); f_a))_a$ and $(h_a^{-n} Graph\, \gamma^{-n}_a)_a$. 
 Then, for every $a$ small, by $(\dagger')$, it holds:
\[|p_y(C_a^n)|\le  (2\epsilon +o(1))|a|^{i+1} .\]
Hence by assumption $(c)$, inequality $(\star)$, it holds 
\[|\eta(a)|=|p_y(C_a)|\le  ((2\epsilon +o(1)) (3/4)^n |a|^{i+1} .\]
Taking $n$ large, it comes that $|\eta(a)|=o(|a|^{i+1})$ and so $\partial_a^{i+1}\eta(0)=0$. This finishes the induction and proves the $C^d$-paratangency property.

\section{Creation of sinks near a homoclinic paratangency}\label{section:creasink}
\paragraph{Sinks}
The following result is well known:
\begin{prop}[Newhouse \cite{Newhouse}]
Let $f$ be a local $C^\infty$-diffeomorphism of $\mathbb R^2$. Suppose there exists a hyperbolic fixed point $\Omega$ which has a quadratic homoclinic tangency at $Q$. If $|\det\,Df(\Omega)|<1$ then there exists $f'$ arbitrarily $C^r$-close to $f$ and $M$ arbitrarily large such that $f'$ has a periodic sink close to $Q$ of period $M$.
\end{prop}
\begin{rema}\label{pointapuit} When $(f_a)_a$ is a (smooth) non-degenerate homoclinic unfolding at $a=0$, such a sink exists for a parameter interval of length $\sim \sigma_a^{-2M}$ which is $\sim \sigma_a^{-M}$ close to $0$, with $\sigma_a$ the unstable eigenvalue of $\Omega_a$ (see \cite{PT93}).
\end{rema}
The concept of homoclinic $C^d$-paratangency enables to produce perturbations for which the periodic sink persists for a parameter interval independent to the size of the perturbation.

In this section we deal with $\infty \ge r> d\ge 0$ or  $\infty \ge r=d \ge 2$. 
 \begin{prop}\label{sinkcreation} 
Let $(f_a)_a$ be a $C^d$-family of local $C^r$-diffeomorphisms of $\R^2$, with $r>d\ge 0$ or $r\ge d\ge 2$. Suppose there exists a persistent hyperbolic periodic point $(\Omega_a)_a$ which  has a homoclinic $C^d$-paratangency at the parameter $a=0$ and at a point $Q$. Suppose that $|\det\,Df(\Omega)|<1$ and, with $|\lambda_a|<1<|\sigma_a|$ the eigenvalues of $Df_a(\Omega_a)$, it holds:
\begin{equation}\tag{$\mathcal D_d$}|\lambda_a|\cdot |\sigma_a|^{d-1}<1.\end{equation}
 Then for every $\mu>0$, there exists $\alpha_0>0$, such that for every $\alpha<\alpha_0$, for every $M$ large, there exists $(f'_a)_a$ so that:
\begin{itemize}
\item $(f'_a)_a$ is $\mu$-close to $(f_a)_a$ in the $C^{d,r}$ topology,
\item  $f'_a=f_a$ for every $a\notin (-2\alpha, 2\alpha)^k$,
\item  $(f'_a)_a$ has a periodic sink of period $M$, close to $Q$, for every $a\in (-\alpha,\alpha)^k$.
\end{itemize}
\end{prop}

To prove this proposition, first we will show the following Lemma:
\begin{lemm}\label{robusthomoclinictangency}
Let $(f_a)_a$ be a $C^d$-family of local $C^r$-diffeomorphisms of $\R^2$. Suppose there exists a persistent hyperbolic periodic point $(\Omega_a)_a$ which  has a homoclinic $C^d$-paratangency at the parameter $a=0$ and the point $Q$. For every $\nu>0$, there exists $\alpha_0>0$, such that for every $\alpha<\alpha_0$, there exists $(f'_a)_a$ so that:
\begin{itemize}
\item $(f'_a)_a$ is $\nu$-close to $(f_a)_a$ in the $C^{d,r}$-topology,
\item  $f'_a=f_a$ for every $a\notin (-2\alpha,2\alpha)^k$,
\item  $(f'_a)_a$ has a homoclinic tangency for every $a\in (-\alpha,\alpha)^k$.
\end{itemize}
\end{lemm}

When $d=0$, both Lemma \ref{robusthomoclinictangency} and Proposition \ref{sinkcreation} are easy. First we notice that we can perturb $(f_a)_a$ to a $C^{\infty,\infty}$-family $(f_a')_a$ in the $C^{0,r}$-topology. By genericity, we can assume that a non-degenerate quadratic homoclinic tangency holds for $(f'_a)_a$ at $a=0$. Let $r_\alpha$ be the odd function equal to $0$ on $[-\alpha,\alpha]$, to $2x-2\alpha$ on $[\alpha, 2\alpha]$ and to the identity on $[2\alpha, \infty)$. 
Let $\rho_\alpha\colon (x_i)_{i}\in \R^k\mapsto \prod_i r_\alpha(x_i)\in \R$.
Such a function is $C^0$-close to the identity when $\alpha$ is small. We notice that $(f'_{\rho_\alpha(a)})_a$ satisfies Lemma \ref{robusthomoclinictangency}. Furthermore, by  Remark \ref{pointapuit}, for $M$ large, there exists $a_M\in \R^k$ close to 0 such that $f'_{a_M}$ has a sink of period $M$. Then the family $(f'_{a_M+\rho_\alpha(a)})_a$ satisfies Proposition \ref{sinkcreation}. 

We can now suppose  $ r>d\ge 1$ or $r =  d\ge 2$.

Before proving the lemma, let us set up some data to study on what depends $\alpha_0$ (the so-called technical bounds mentioned in the plan of the proof).

By Proposition \ref{cartecool}, in the coordinates given by  a $C^{d}$-family of $C^r$-charts $(\phi_a)_a$ from a neighborhood of $\Omega_a$ onto  $D:= (-\eta,\eta)^2$, we have for every $a$:
\[\Omega_a=0,\quad W^u_\eta(\Omega_a; f_a)=(-\eta,\eta)\times \{0\},\quad \text{and}\quad W^s_\eta(\Omega_a; f_a)=\{0\}\times (-\eta,\eta).\]

By definition of quadratic tangency, there exists a point $P_0=(p_0,0)\in D$ sent by an iterate $f^N_0$ to a point $Q_0=(0,q_0)\in D$,
and for every $a$ small, there exists a neighborhood $D_P$ of $P$ in $D$ sent by $f_0^N$ into $D$, there exist $(A_a),(B_a)_a$ of class $C^{d,r}$, so that  $f_a^N|D_P$ has the following form:
\[ \phi_a \circ f_a^N\circ \phi_a^{-1}|D_P\colon (p_0+x, y)\in D_P\mapsto  (0, q_0)+ (A_a(x, y), B_a(x, y))\in D\; ,\]      
where $(A_0,B_0)$ is of class $C^r$,  and satisfies at $a=0$:
\begin{equation} \tag{$\mathcal T$} 0=A_0(0)=B_0(0)=\partial_x A_0(0) \quad\text{and} \quad \partial^2_x A_0(0)\not=0.\end{equation}

By $(\mathcal T)$ and Fact \ref{prametrizationtangent}, for $a$ small, there exists   $a\mapsto c_a$ of class $C^{d-1}$ satisfying $c_0=0$ and 
\begin{equation} \tag{$\mathcal T'$} \partial_x A_a(c_a,0)=0, \quad \forall a\; \text{small}.\end{equation}
Moreover the function $C\colon a\mapsto A_a(c_a,0)$ is of class $C^d$. By definition of $C^d$-paratangency, it holds:
 \begin{equation} \tag{$\mathcal P$} 0=C(0)=\partial_a C(0)=\cdots = \partial_a^d C(0).\end{equation}

For every $j\in \{0,\dots, N\}$, let $P_a^j:= f^j_a (P_0)$.
Let $\theta>0$ be so that $P_a^{N-1}$ is $\theta$-distant to both $D$ and $(P_a ^j)_{0\le j\le N-2}$.
Let $(\Psi_a)_a$ be a $C^{d,r}$-family of charts, so that each $\Psi_a$ sends a $\theta$-neighborhood of $P_a^{N-1}$ onto the unit ball of $\R^2$ so that 
 $P_a^{N-1}$ is mapped to $0$ and a local unstable manifold of $P_a^{N-1}$ is mapped onto $(-1,1)\times\{0\}$.
We remark that we have the form:
\[\phi_a \circ f_a\circ \psi_a^{-1}|D_P \colon (x, y)\in D_P\mapsto  (0, q_0)+ (A'_a(x, y), B'_a(x, y))\in D,\]  
with $(A'_a)_a$ and $(B'_a)_a$ of class $C^{d,r}$, satisfying $(\mathcal T)$ at $a=0$. Moreover there exists 
$a\mapsto c_a$ of class $C^{d-1}$ satisfying $c'_0=0$ and $(\mathcal T')$. By definition of $C^d$-paratangency, the map 
  $C'\colon a\mapsto A(c_a',0)$ is of class $C^d$ and satisfies $(\mathcal P)$. 
\begin{defi}[$(\theta,U,\nu )$-$C^{d,r}$-paratangency]
Let $U$ be a bound on the $C^{d,r}$-norms of $(\phi_a)_a$, $(\psi_a)_a$, $(\phi_a^{-1})_a$, $(\psi_a^{-1})_a$,   $(A'_a)_a$ and $(B'_a)_a$. Let $\nu\in C^0(\R)$ be a modulus of continuity  of $\partial_a^d C'(a)$, in other words, it holds: $|\partial^d_a C'(a)|\le \nu(a)$. Then we say that $(\Omega_a)_a$ has a \emph{$(\theta,U,\nu )$-$C^{d,r}$-paratangency}.  
\end{defi}
\begin{Claim}\label{dpeendence} An upper bound of the value of $\alpha_0$ in Lemma \ref{robusthomoclinictangency} and Proposition \ref{sinkcreation}
depends only on $\mu$ and $(\theta, U,\nu)$, and not on $N$.
\end{Claim} 
We remark that this claim holds true in the case $d=0$. Now we study the remaining case $ r>d\ge 1$ or $r= d\ge 2$. 
\begin{proof}[Proof of Lemma \ref{robusthomoclinictangency} and Claim \ref{dpeendence}]
Let $\phi\in C^\infty(\R,[0,1])$ be a function with support in $[-1,1]$ and equal to $1$ on  $[-1/2,1/2]$.
Let $\rho \colon (x_i)_i\in \R^k\mapsto \prod_i \phi(x_i) \in \R$.
 For $\alpha>0$ and $a\in \R^k$, we consider the perturbation $f_a'$ of $f_a$ equal to $f_a$ outside of $B(P_a^{N-1},\theta)$ or $a\notin (-2\alpha,2\alpha) ^k$, and otherise whose restriction to $B(P_a^{N-1},\theta)$ is:
\[\phi_a \circ f_a'\circ \psi_a^{-1}(z)=\phi_a \circ f_a\circ \psi_a^{-1}(z)+ \phi(\|z\|) \rho(\frac {a}{2\alpha})\cdot (-C'(a),0).\]
It is easy to see that  $(f'_a)_a$ has a homoclinic tangency for every $a\in (-\alpha,\alpha)^k$. Furthermore $(f_a')_a$ is a $C^{d}$-family of $C^r$-local diffeomorphisms.

 Let us show that the family $(f_a')_a$ is close to $(f_a)_a$ for the $C^{d,r}$-topology. For this end, it suffices to show that the function $\rho_\alpha:=a\mapsto \phi(\|a\|/2\alpha)C'(a)$ is of  $C^d$-norm small when $\alpha$ is small. Indeed, for all $a\in (-2\alpha,2\alpha)^k$, $\forall i\le d$, it holds: 
\[ |D^k\rho_\alpha(a)|= |\sum_{i=0}^k C_k^i (2\alpha)^{-i} (D^i\phi)(\|a\|/2\alpha) D^{k-i}C'(a)|= \sum_{i=0}^k \alpha^{-i} o(|\alpha|^{d-k+i})=o(|\alpha|^{d-k})
.\]  
\end{proof}
\begin{proof}[Proof of Proposition \ref{sinkcreation}] Since the case $d=0$ is already done, we assume $r>d>0$ or $r=d\ge 2$.
To prove this Proposition, we continue with the notations of the above Lemma. We can suppose that for the charts defined above, we have the form:
\[\phi_a \circ f'_a\circ \psi_a^{-1}|D_P \colon (x, y)\in D_P\mapsto  (0, q_0)+ (A'_a(x, y), B'_a(x, y))\in D,\]  
with $(A'_a)_a$ and $(B'_a)_a$ of class $C^{d,r}$ satisfying $(\mathcal T)$ at $a=0$, 
and existence of $a\mapsto c'_a$ of class $C^{\min(r-1,d)}$, with $c'_0=0$, and 
\begin{equation} \tag{$\mathcal P_\infty$} 0=A'_a(c'_a,0)=\partial_x A'_a(c'_a,0), \quad \forall a\in (-\alpha,\alpha)^k.\end{equation}
 
We notice that if $r>d$, the map $a\mapsto c'_a$ is of class $C^d$. Nevertheless, when $r=d\ge 2$, the map $a\mapsto c'_a$ is of class $C^{d-1}$. This causes  an extra difficulty in the proof. Hence, we assume now that $r>d$, and so that $a\mapsto c'_a$ is of class $C^d$, and at the end of the proof, we give a trick to go back to this case.

Let us now do a special perturbation, for which it will appear a sink. For this end, we are going to define a family of curves. 

By \cite[Prop. 2.10]{MisuRen} and $(\mathcal D_d)$, for $\eta>0$ small enough, there exists a line field $(a,z)\in (-1,1)\times D\mapsto e_a(z)\in P\mathbb R$ of class $C^{d}$ so that for every $a\in (-1,1)^k$ and $z\in D\cap f'^{-1}_a(D)$, the differential $D_zf'_a$ sends $e_a(z)$ onto $e_a(f'_a(z))$. Moreover $e_a$ is tangent to $W^s_\eta(\Omega; f'_a)$.

We recall that $(f_a^{N-1}|D_P)^{-1}\psi_a^{-1} (c'_a,0)$ belongs to $W^u_\eta(\Omega_a; f_a)$. It is sent by $\phi_a$ to a point $(p_a,0)$ close to $(p_0,0)$. For $n\ge 0$, we put $(p_a^n,0):= \phi_a \circ (f_a|D)^{-n}\circ \phi_a^{-1}(p_a,0)$.
The curve $(p_a^n)_a$ is of class $C^d$. By inclination Lemma \ref{inclination} applied to $(f_a|W^u_\eta(\Omega_a; f_a))^{-1}$, the curve $(p_a^n)_a$ is $C^d$-close to $(\Omega_a)_a$ when $n$ is large. 

For every $a$, let $W^n_a$ be the curve of $D$ which is tangent to $e_a$ and contains $(p^n_a,0)$. Let $w^n_a\in C^{d}((\eta,\eta),\R)$ be such that:
\[W^n_a=\{(w^n_a(y),y)\colon y\in (-\eta,\eta)\}.\]

Also we notice that $\hat W^n:=\sqcup_a \{a\}\times  W^n_a$ is of class $C^d$: it is given by integrating in $\R^3$  the $C^{d}$-vector field $(a,z)\mapsto (0, e_a(z))$ along  the graph of the $C^d$-function $a\mapsto p^n_a$. Hence the map $(a,y)\mapsto w^n_a(y)$ is of class $C^d$. As $e_a|\{0\}\times (-\eta,\eta)= \{0\}\times \R$ and $(p_a^n)_a$ is $C^d$-close to $0$, the map $(a,y)\mapsto w^n_a(y)$ is $C^d$-small for $n$ large.

Consequently, the function $a\mapsto w^i_a\circ B'_a(c'_a,0)$ is of class $C^d$, with norm small for $i$ large.
 
For $n$ large, we consider the perturbation $(f''_a)_a$ of $(f_a)_a$ with $f''_a$ equal to $f_a$ outside of $B(P_a^{N-1},\theta)$ or $a\notin (-2\alpha,2\alpha)^k$, and whose restriction to $B(P_a^{N-1},\theta)$ is:
\[\phi_a \circ f_a''\circ \psi_a^{-1}(z)=\phi_a \circ f'_a\circ \psi_a^{-1}(z)+\phi(\|z\|) \rho(\frac {a}{2\alpha})\cdot (w^n_a\circ B'_a(c'_a,0),0),\]
with $\phi$ and $\rho$ defined in the proof of Lemma \ref{robusthomoclinictangency}.

Likewise $(f_a'')_a$ is of class $C^{d,r}$, and for $n$ large, it is close to $(f_a')_a$ (and so to $(f_a)_a$) for the $C^{d,r}$-topology.
We remark that $f''_a|D= f'_a|D=f_a|D$ for every $a$.

Now the map $(f''_a)$ does not satisfy ($\mathcal T$). But if $n$ is large enough, for every $a\in(-\alpha,\alpha)^k$, $f_a''$ has a sink of period $N+n$ by the following claim proved below:
\begin{Claim} \label{unpuit} The following map, well defined on a neighborhood $D_p'$ of $P_a$ for every $a\in [-\alpha,\alpha]^k$, has a sink :
\[\phi_{a\,n}:=(f_a|D)^n\circ f_a''^{N}|D_p= (f_a''|D)^n\circ f''^N_a|D_p.\]
\end{Claim}
This finishes the proof of the Proposition in the case $r>d\ge 1$.

Now let us treat the case $r=d\ge 2$. For this end, by density of the $C^{d,r+1}$-families in the $C^{d,r}$-topology, there exists 
a $C^{d,r+1}$-family $(\bar A_a)_a$ which is  $C^{d,r}$-close to $(A'_a)_a$. The critical point $c_a'$ of $x\mapsto A_a(x,0)$ persists as a critical point $\bar c_a$ of $A_a$.
We observe that $a\mapsto \bar c_a$ is $C^{d-1}$-close to $a\mapsto c'_a$. Hence $\bar C(a)=\bar A_a(\bar c_a,0)$ has its derivative $\partial_a \bar A_a(\bar c_a,0)\cdot \partial_a c_a $ which is $C^{d-1}$-close to $\partial_x C'=0$. Consequently, the function $\bar C$ is $C^d$-close to $0$ when $(\bar A_a)_a$ is close to $(A'_a)_a$. Therefore, the family $(\bar A'_a)_a$ defined by  $\bar A'_a= \bar A_a-\bar C(a)$ is of class $C^{d,r+1}$, is $C^{d,r}$-close to $(A'_a)_a$ and satisfies $(\mathcal P_\infty)$ at $(\bar c'_a,0)$. Thus we can apply the above argument with the map $\bar f'$ defined by
\[\phi_a \circ \bar f'_a\circ \psi_a^{-1}|D_P \colon (x, y)\in D_P\mapsto  (0, q_0)+ (\bar A'_a(x, y), B'_a(x, y))\in D.\]
Note that  it is not a problem to patch $\bar f'$ to $f'$ outside of $D_P$.  

Nonetheless, the  function $a\mapsto \bar c_a$ has $C^d$-norm a priori large, the line field as well, and so they define a the function $(a,y)\mapsto w^0_a(y)$ of a priori the $C^d$-norm large. This is actually not a problem, since the $C^d$-norm of $(a,y)\mapsto w^n_a(y)$ remains small for $n$ large.

\end{proof}
Before proving this Claim, let us generalize this proposition to a more general context which will appear in the proof of Theorem \ref{theo1}. 

Let  $r> d\ge 0$ or $r=d\ge 2$, let $M$ be a surface and let  $(\Omega_a)_a$ be a curve of hyperbolic fixed points for a $C^{d}$-family $(f_a)_a$ of $C^r$-local diffeomorphisms of $M$.
Let us suppose that there exists a sequence of points $P_{j\, a}\in W^u(\Omega_a; f_a)$ endowed with local unstable manifold $W^u_{\loc}(P_{j\, a}; f_a)$ so that the family $((W^u_{\loc}(P_{j\, a}; f_a))_a)_j$ converges in the $C^{d,r}$-topology to a certain $(W^u_{\loc}(P_{\infty\, a}; f_a))_j$. 
\begin{defi} The fixed point $(\Omega_a)_a$ has a \emph{quasi-homoclinic $C^d$-paratangency} if  $(W^u_{\loc}(P_{\infty\, a}; f_a))_a$ is $C^d$-paratangent to $(W^s(\Omega_a; f_a)_a)_a$. 
\end{defi}
Likewise we can define the concept of $(\theta, U,\nu)$-quasi-homoclinic $C^{d,r}$-paratangency. Again, seen via a $C^{d}$-family of $C^r$-charts $(\phi_a)_a$ from a neighborhood of $\Omega_a$ onto  $D:= (-\eta,\eta)^2$, we can assume that for every $a$:
\[\Omega_a=0,\quad W^u_\eta(\Omega_a; f_a)=(-\eta,\eta)\times \{0\},\quad \text{and}\quad W^s_\eta(\Omega_a; f_a)=\{0\}\times (-\eta,\eta).\]

By replacing $(P_{j\, a})_j$ and $(W^u_{\loc}(P_{j\, a}; f_a))_j$ by their images by an iterate $f^k_a$, we can suppose that $(f_a(W^u_{\loc}(P_{\infty\, a}; f_a)))_a$ is $C^d$-paratangent to $(W^s_\eta(\Omega_a; f_a))_a$ at $a=0$. 


Let $(\psi_a)_a$ be a $C^d$-family of $C^r$-charts so that each $\psi_a$ sends a $\theta$-neighborhood of $P_{\infty\, a}$ onto the unit ball of $\R^2$, maps $W^u_{\loc}(P_{\infty\, a}; f_a)$ onto $(-1,1)\times\{0\}$ and $P_{\infty\, a}$ to $0$. We remark that we have the form:
\[\phi_a \circ f_a\circ \psi_a^{-1}|D_P \colon (x, y)\in D_P\mapsto  (0, q_0)+ (A_a(x, y), B_a(x, y))\in D,\]  
with $(A_a)_a$ and $(B_a)_a$ of class $C^{d,r}$ which satisfy, $A_0(0)=\partial_x A_0(0)=0$ if $r\ge 1$. If $r\ge 2$, these families of functions satisfy moreover condition $(\mathcal T)$. 
Moreover there exists  $a\mapsto c_a$ of class $C^{d-1}$ satisfying $c_0=0$ and $(\mathcal T')$. By definition of $C^d$-paratangency, the map   $C\colon a\mapsto A(c_a,0)$ of class $C^d$ and satisfies $(\mathcal P)$.

For every $P_{j\, a}$, let $(P_{j\, a}^k)_{k\le 0}$ be the preorbit of $P_{j\, a}$ which defines $W^u_{\loc}(P_{j\, a}; f_a)$:
\begin{itemize} 
\item  $ P_{j\, a}^0= P_{j\, a}$ and for $k<0$, $f_a(P_{j\, a}^k)=  P_{j\, a}^{k+1}$,
\item  for $k<0$ large, $P_{j\, a}^k$ belongs to $W^u_\eta(\Omega_a; f_a)$, and $W^u_{\loc}(P_{j\, a}; f_a)$ is the image by $f_a^{-k}$ of a segment of $W^u_\eta(\Omega_a; f_a)$ containing $P_{j\, a}^k$.
\end{itemize}
 
\begin{defi}[$(\theta, U,\nu)$-quasi-homoclinic $C^{d}$-paratangency]
Suppose that $P_{\infty\, a}$ is $\theta$-distant to $\{f_0^k(P_{\infty\, 0})\colon  k\ge 1\}\cup \bigcup_j \{P_{j\, a}^k\colon k<0\}$ for every $a$.   

If $U$ is a bound on the $C^{d,r}$-norms of $(\phi_a)_a$, $(\psi_a)_a$, $(\phi_a^{-1})_a$, $(\psi_a^{-1})_a$,   $(A'_a)_a$ and $(B'_a)_a$. Let $\nu\in C^0(\R)$ be the  modulus of continuity  of $\partial_a^d C$ at $a=0$. We say  that $(\Omega_a)_a$ has a \emph{$(\theta, U,\nu)$-$C^{d,r}$-paratangency}.  
\end{defi}

\begin{prop}\label{quasisinkcreation}
If $(\Omega_a)_a$ has a $(\theta, U,\nu)$-quasi-homoclinic $C^{d}$-paratangency, and $Df_a(\Omega)$ satisfies $(\mathcal D_d)$ and  $|\det\,Df_a(\Omega)|<1$, then 
 for every $\mu>0$, there exists $\alpha_0>0$, depending only on $(\theta, U,\nu)$, so that for every $\alpha<\alpha_0$, for every $M$ large, there exists $(f'_a)_a$ satisfying:
\begin{itemize}
\item $(f'_a)_a$ is $\mu$-$C^{d,r}$-close to $(f_a)_a$,
\item  $f'_a=f_a$ for every $|a|>2\alpha$,
\item  $(f'_a)_a$ has a periodic sink of period $M$, close to $P_{\infty\, a}$, for every $a\in [-\eta,\eta]$.
\end{itemize}
\end{prop}
\begin{proof} In view of Proposition \ref{sinkcreation}, it suffices to perturb $(f_a)_a$ to a family $(f'_a)_a$ which has $(\theta/2, 2U,2\nu)$-$C^{d,r}$-paratangency.
For every large $j$, the exists a $C^{d,r}$-family $(\tau_a)_a$ of diffeomorphisms $\tau_a$ of $\R^2$ which sends $W^u_{loc} (P_{j\, a})$ onto $W^u_{loc} (P_{\infty\, a})$,
equal to the identity outside the $\theta$-neighborhood of $P_{\infty\, a}$ and which is $C^{d,r}$-close to $(id_{\R^2})_a$ when $j$ is large.


 
Let $\phi\in C^\infty(\R,[0,1])$ be a function with support in $[-1,1]$ and equal to $1$ on  $[-1/2,1/2]$. Put $\rho\colon (a_i)_i\in \R^k\mapsto \prod_i \phi(a_i)$.
We define:
\[f'_a(z)= \left\{\begin{array}{cl}
f_a(z)+ \rho(\frac {a}{2\alpha}) (f_a \circ \tau_a(z)-f_a(z))& z\in B(P_{\infty , a}, \theta),\\
f_a(z) &\text{otherwise.}\end{array}\right.\] 
We notice that $(f'_a)_a$ is a family of local diffeomorphisms $C^{d,r}$-close to $(f_a)_a$ when $j$ is large.
The map $f'_a$ is equal to $f_a$ on the complement of $B(P_{\infty , a}, \theta)$ or $a\notin (-2\alpha,2\alpha)^k$.
We remark also that $f'_a$ has a $(\theta, 2U,2\nu)$-$C^{d,r}$-paratangency.


\end{proof}

\begin{proof}[Proof of Claim \ref{unpuit}] This claim just needs to be proved in the case $r>d>0$ or $r=d\ge 2$. 
We are going to work in another family of charts in which $f_a|D$ has a skew product form. This family of charts is less regular than the previous ones, but it does not matter since we do not need to perturb the dynamics anymore.

Taking $\eta$ smaller if necessary, the local unstable manifold $W^u_\eta(\Omega; f_a)=(-\eta,\eta)\times\{0\}$ is transverse to $e_a$. A flow box for this line field gives a chart of a neighborhood of $\Omega$, in which $L_a:=f_a|D\cap f_a^{-1}(D)$ has the following skew product form:
\[L_a\colon (x,y)\mapsto (f_a(x,0), f_a(x,y)).\]
Put $g_a(x):=f_a(x,0)$ and $h_{a \, x}(y):= f_a(x,y)$, with $(g_a)_a$ of class $C^{d,r}$ and $(a,y)\mapsto h_{a \, x}(y)$ of class $C^d$. Note also that  $(a,y)\mapsto \partial_y h_{a \, x}(y)$ is of class $C^d$. In these notations, for $n\ge 1$:
\[L_a(x,y)= (g_a(x), h_{a \, x}(y)),\; L_a^n(x,y)= (g^n_a(x), h^n_{a \, \overline x}(y)),\]
 with $ h^n_{a \, \overline x}=h_{a \, g^{n-1}_a(x)}\circ \cdots \circ h_{a\, x}$. 

In these coordinates observe that $f_a$ has the form:
\[f_a''^N\colon (p_a+x,y)\in D_P\mapsto (p^n_a+A''_a(x,y),B''_a(x,y)),\]
with for every $a\in (-\alpha,\alpha) ^k$:
\[ 0=A''_a(0)=\partial_x A''_a(0),\quad g^n_a(p^n_a)=p_a.\]

Let $\kappa>1$ be so that for every $a$, with $\sigma'_a=\kappa^2 \sigma_a$ and $\lambda'_a=\kappa^2 \lambda_a$, it holds $|\sigma_a'\lambda_a'|<1$.	By taking $\eta$ and $\kappa$ smaller if necessary, we can assume that for every $(x,y)\in D$:
 \[|\partial_x g_a(x)|\le \kappa |\sigma_a|=  |\sigma'_a|/\kappa\quad \text{and}\quad |\partial_y h_{a\, x}(y)|\le \kappa |\lambda_a|=|\lambda'_a|/\kappa.\]
   
Let us compute the expression of $\phi_{a\,n}:=(f_a|D)^n\circ f_a''^{N}|D_p$:
\[\phi_{a\,n}\colon  (p_a+x, y)\in D_P\mapsto  (g_a^n(A''_a(x, y)+p^n_a),h^n_{a \, \overline z}( B''_a(x, y))),\]
 with $z:=A''_a(x, y)+p^n_a$.

Let us study the derivatives of the first coordinate of $\phi_{a\,n}$. Let $C$ be a constant independent of $n$ and $(x,y)$. Using  the volume contraction of $D_\Omega f_a$, we show below the following: 
\begin{lemm}\label{deriveepartielle}
For $(x,y)$ in the set $B:=p_a+(- |\sigma_a'^{-n}|,|\sigma_a'^{-n}|)\times (- |\sigma_a'^{-3n}|,|\sigma_a'^{-3n}|)$,  the Jacobean matrix of $D_{(x,y)}\phi_{a\, n}$ has its entries smaller than those of:
\[C\kappa^{-n}\left[\begin{array}{cc} |\sigma_a'|^n\cdot |x| & |\sigma_a'|^n\\ n|\lambda_{a}'|^n&n|\lambda_{a}'|^n\end{array}\right]\le 
C \left[\begin{array}{cc} \kappa^{-n} & \kappa^{-n}|\sigma_a'|^n\\
 n\kappa^{-n}|\lambda_{a}'|^n& n\kappa^{-n}|\lambda_{a}'|^n\end{array}\right],\]
with $C$ a constant independent of $n$.
\end{lemm}

We remark that for the norm $N(u,v)=|u|+|\sigma_a'^{n} v|$, the subordinate norm of the latter matrix is at most $C \kappa^{-n}(1+Cn|\sigma_a'|^{n}|\lambda'_a|^{n})$,
 which is small for $n$ large. Thus for this norm, the differential of $\phi_{a\,n}$ is (very) contracting. 
As $\phi_{a\,n}(p_a)=(p_a, h^n_{a \, \overline z}( z))$ is $\lambda_a'^n$-close to $(p,0)$, the set $B$ is sent into itself by $\phi_{a\,n}$.
 Therefore, there is a sink of period $n+N$ in $B$.\end{proof}
\begin{proof}[Proof of Lemma \ref{deriveepartielle}]
We recall that $r\ge 2$. It holds:
\[|\partial_y g_a^n(A''_a(x, y)+p^n_a)|\le C \kappa^{-n} |\sigma_a'^n|,\]
 As $\partial_x A''_a(0)=0$, it comes:
\[|\partial_x (g_a^n(A''_a(x, y)+p^n_a)|\le  \kappa^{-n} |\sigma_a'^n|\cdot |x|.\]
To study the derivative of the second coordinate $h^n_{a \, \overline z}( z)$ of $\phi_{a\,n}$, we bound the derivatives of $h_{a\, \overline x}^n$:
  \[|\partial_y h^n_{a \, \overline x}(y)|\le  \kappa^{-n} |\lambda'^n_a|\]
  \[\partial_x h^n_{a \, \overline x}(y)=\sum_{k=0}^{n-1}\partial_y (h_{a\, x_{n-1}} \circ \cdots \circ h_{a\, x_{k+1}})(\partial_x h_{a\, x_k})(h^k_{a\, \overline x}(y)),\; \text{with}\; h^k_{a\, \overline x}=(h_{a\, x_{k-1}}\circ\cdots \circ  h_{a\, x_{0}})\]
	We notice that $|\partial_y (h_{a\, x_{n-1}} \circ \cdots \circ h_{a\, x_{k+1}})|$ is bounded by $(\lambda_a'/\kappa)^{n-k-1}$, whereas $h^k_{a\, \overline x}(y)$ is bounded by $(\lambda_a'/\kappa)^k$. As $\partial_x h_{a\, x_k}(0)=0$ and $\partial_y\partial_x h_{a\, x_k}$ are well defined and continuous, we get that  $\partial_x h_{a\, x_k}(h^k_{a\, \overline x}(y))$ is dominated by $(\lambda_{a}'/\kappa)^k$. Consequently:
	\[|\partial_x h^n_{a \, \overline x}(y)|\le C n \kappa^{-n} |\lambda_{a}'|^n.\]
As $DB''_a(x, y)$ and $Dz$ are bounded, it comes: 
\[|\partial_x h_{a \, \overline z}^n(B''_a(x, y))|\le  C n \kappa ^{-n}  |\lambda_{a}'|^n,\; 
|\partial_y h_{a \, \overline z}^n(B''_a(x, y))|\le  C n \kappa ^{-n}  |\lambda_{a}'|^n.
\]
\end{proof}
\section{Proof of main Theorem \ref{main}: The local diffeomorphism case}\label{section:proofmain}
In \cite{DNP06}, it is shown the existence of a Baire generic set of $C^r$-diffeomorphisms with infinitely many sinks by 
coupling a blender and a dissipative saddle point, for $r\ge 1$. We push forward this construction to obtain Theorem \ref{main}, by coupling a parablender and a dissipative saddle point.

We recall that in \textsection \ref{defparablender}, we defined $d'$ as the dimension of $\{P\in P[X_1,\dots,X_k]:\; P(0)=0,\; deg\, P\le d\}$ (hence $d'=d$ for $k=1$). Also we defined:
\[ \mathring f\colon (x,y)\in \mathbb R/6\Z \times \R\mapsto (Q^{d'+1}(x), 2y/3+\rho(x)/3)\in \mathbb R/6\Z \times \R,\]
with $Q\colon x\in \mathbb S^1:= \mathbb R/6\mathbb Z\mapsto 4x-3\in \mathbb S^1$ and with $\rho\in C^\infty (\mathbb S^1, [0,1])$ a smooth function on $\mathbb S^1$ equal to $1$ on $[1/2,1]$ and $-1$ on $[-1,-1/2]$ and with support in $[-3/2, 1/4]\cup [1/4, 3/2]$. For every $C^r$-perturbation of $f$ of $\mathring f$, we also defined:
\[f_{\epsilon\, a}(z)=  f(z)+(0,\epsilon P_{\delta}(a)),\quad \text{if } z\in \tilde Y_{\delta}.\]

We are going to deform $\mathring f$ to create a dissipative homoclinic tangency.

We remark that $Q(3)=3$ and $\rho(3)=0$. Then the map $\mathring f$ fixes the point $\mathring \Omega= (3,0)$. We notice that the eigenvalues of 
$\mathring \Omega$ are $\mathring \sigma:=4^{d'+1}$ and $\mathring \lambda= (2/3)^{d'+1}$. Let us deform $\mathring f$ so that this hyperbolic  fixed point persists as a  volume contracting one which satisfies condition $(\mathcal D_d)$.

For this end, let $\phi\in C^{\infty}(\mathbb S^1,[0,1])$ be a function with support in $[-1,1]$ and equal to $1$ on $[-1/2,1/2]$.

 Now we consider the map:
 \[\mathring f'\colon (x,y)\in \R/6\Z\times \R \mapsto \mathring f(x,y)+ \phi(x-3) ((Q^{d'+1}(x),4^{-(d'+2)^2} y)-\mathring f(x,y))\; .\]
We notice that $\mathring \Omega:= (3,0)$ is a hyperbolic fixed point of $\mathring f'$, and it satisfies $|\det\;D_{\mathring \Omega} \mathring f'|<1$ and condition $(\mathcal D_d)$.

We remark that the point $(0,0)$ is sent by $\mathring f'$ to $\mathring \Omega$. Consequently, we can perform a deformation of $\mathring f'$ to construct a  map $\mathring g$ so that for $\eta\in (0, 4^{-d-2})$ small it holds:
\begin{itemize}
\item On the complement of $[-2\eta,2\eta]\times \R$, the map   $\mathring g$  is equal to $\mathring f'$.
\item The point $(0,0)$ is sent by $\mathring g$ to the point $Q:=(3,1)\in W^s_{1}(\mathring \Omega;\mathring g)=\{3\}\times[-1,1]$.
\item The preimage of $W^s_{1}(\Omega; \mathring g)$ by $\mathring g|[-\eta,\eta]\times[-2,2]$ is the graph of $x\in[-\eta ,\eta]\mapsto 2x^2$.
\end{itemize}
\begin{rema}\label{rema1} For every $x\in \mathbb S^1$, the restriction $\mathring f|\{x\}\times \R$ is injective.\end{rema}
We draw the construction at Figure \ref{preuvemain}.
\begin{figure}[h]
    \centering
        \includegraphics{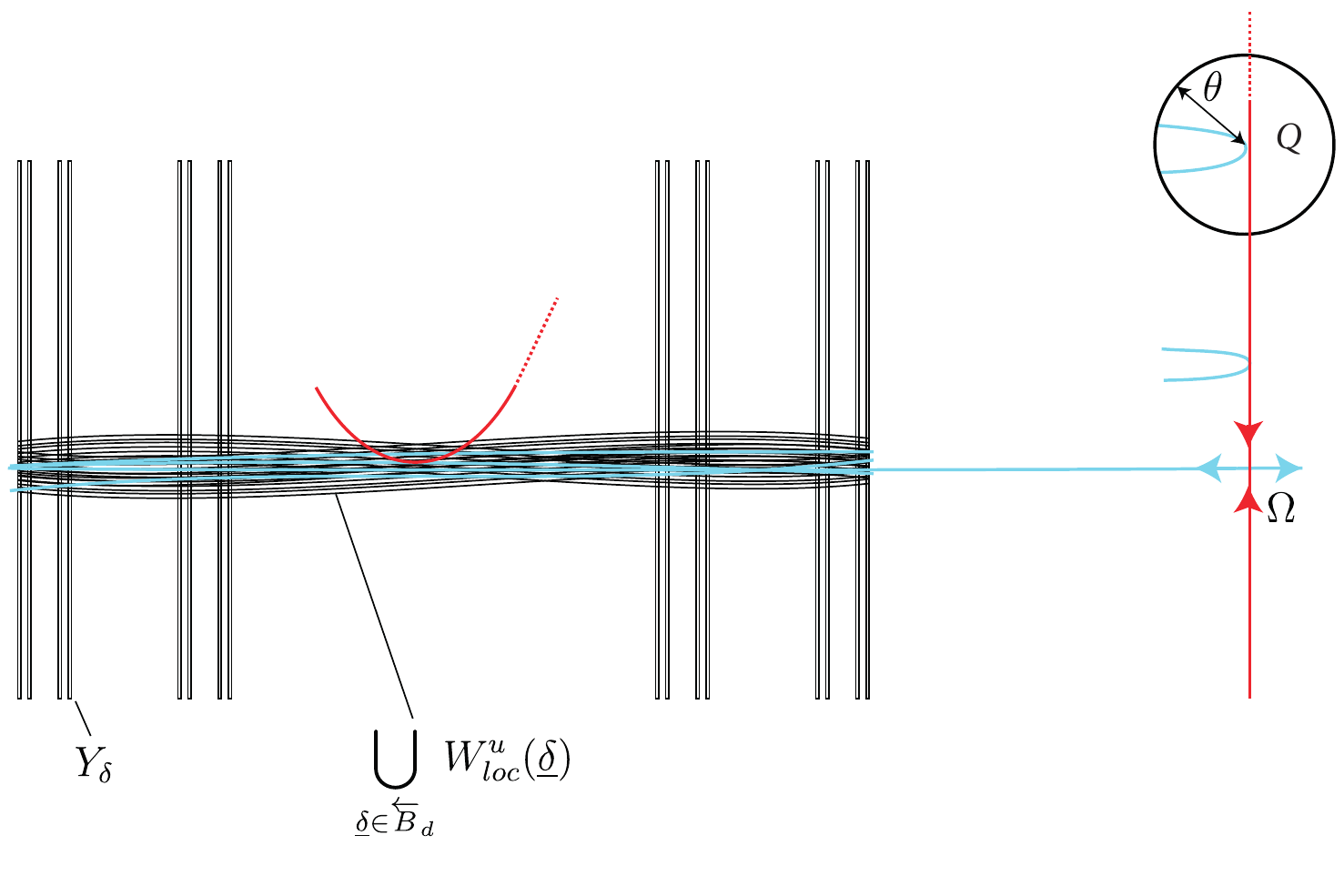}
    \caption{Coupling of a parablender with a saddle point.}
    \label{preuvemain}
\end{figure}
For any $C^r$-perturbation $g$ of $\mathring g$ and $\epsilon >0$ small, we consider the family $(g_{a\, \epsilon})_a$ so that: 
\[g_{\epsilon\, a}(z)=  \left\{\begin{array}{cl}
g(z),& \text{if}\;  z\in [-1/4,1/4]\times [-2,2],\\
g(z)+(0,\epsilon P_{\delta}(a)),& \text{if } z\in \tilde Y_{\delta},\; \delta\in \Delta_{d'}.\end{array}\right.\]


We recall that $\mathring g|[-2,-1/4]\sqcup [1/4,2]\times [-2,2]$ is equal to $\mathring f$. Hence $\mathring g$ and its perturbation $g$ still satisfy the  fundamental property of the parablender \ref{Ml2k}. 
Thus there exist $\alpha>0$,  a neighborhood $U_0$ of $\mathring g$, and for every $\epsilon>0$ small, there exists a neighborhood $U_\epsilon$ of $\{(g_{\epsilon\, a})_a\colon g\in U_0\}$ in the $C^{d,r}$-topology, so that for every $(g_{ a})_a \in U_\epsilon$, for every $a\in [-\alpha,\alpha]^k$, there exists $\underline \delta^a\in \Delta_{d'}^\Z$ so that   $W^u(\underline \delta^a; g_a)$ is $C^d$-paratangent to the parabola $(g_a|[-\eta,\eta]\times[-2,2])^{-1}( W^s_{1}(\Omega_a; g_a))$, with $\Omega_a$ the fixed point of $g_a$ equal to the hyperbolic continuation of $\mathring \Omega$.

We remark that the tangency point, denoted by $P_{\infty\, a}$, has its image by $g_a$ which is close to $Q=(3,1)$.

\begin{fact}
For every $(g_a)_a\in U_\epsilon$, for every $a\in (-\alpha,\alpha)^k$, the hyperbolic fixed point $\Omega(g_a)$ has a quasi-homoclinic $C^d$-paratangency.
\end{fact}
\begin{proof}
We consider the following preimages of $P_{\infty\, a}$.
\[P_{\infty\, a}^0:= P_{\infty\, a},\quad P_{\infty\, a}^{i}:= (g_a|\tilde Y_{\delta^a_{i}})^{-1}( P_{\infty\, a}^{i+1}),\; \forall i<0.\]
We notice that the sequence $(P_{\infty\, a}^{j})_{j<0}$ remains in $\cup_{\delta} \tilde Y_\delta$. Also, $W^u_{loc}(\underline \delta^a, g_a)$ is a local unstable manifold of the preorbit  $(P_{\infty\, a}^{j})_{j\le 0}$.

On the other hand, a local unstable manifold $W^u_{3/4^{d'+1}}(\Omega_a;g_a)$ is close to $(-3/4^{d'+1}+3,3/4^{d'+1}+3)\times\{0\}$. As $d'\ge 0$, it is close to be contained in $[9/4, 15/4]\times \{0\}$. We recall that at the neighborhood of this segment, the map $g_a$ is close to $(x,y)\mapsto (Q^{d'+1}(x), h(x)\cdot y) $ for a function $h$ with values in $[-3/4, 3/4]$. Consequently, the image $W^u_{3}(\Omega_a;g_a)$ of $W^u_{3/4^{d'+1}}(\Omega_a;g_a)$ by $g_a$ is close to $(0,6)\times \{0\}$.
For every $j\ge 1$, we consider the following local unstable manifold:
\[Y_0\cap g_a^j \left(  W^u_{3}(\Omega_a;g_a)\cap \bigcap_{n=1}^{j} g_a^{n-j}( \tilde Y_{\delta^a_{-n}})\right)=: W^u_{loc}(P_{j\, a}; g_a)\; .\]  
By Lemma \ref{inclination}, this unstable manifold is $C^{d}$-close to $W^u_{loc} (\underline \delta^a; g_a)$ when $j\ge 1$ is large. Moreover the family $(W^u_{loc}(
P_{j\, a}; g_{a'}))_{a'}$ is $C^{d,r}$-close to $(W^u_{loc} (\underline \delta^a; g_{a'}))_{a'}$, with 
\[ W^u_{loc}(P_{j\, a}; g_{a'}) := Y_0\cap  g_{a'}^j \left(  W^u_{3}(\Omega_{a'};g_{a'})\cap \bigcap_{n=1}^{j} g_{a'}^{n-j}( \tilde Y_{\delta^a_{-n}})\right)\; .\]  
 For every $j\ge 1$, this manifold is $C^r$-close to be a segment of the form $[-1,1]\times \{y_j\}$ for $y_j\in [-3/2,3/2]$. Consequently there exists a point $P_{j\, a}\in W^u_{loc}(P_{j\, a}; g_a)$ with the same $x$-coordinate as $P_{\infty\, a}$. In particular $(P_{j\, a})_j$ converges to $P_{\infty\, a}$.
 \end{proof}
Let $\theta= 1/4$. We notice that for every $a\in [-\alpha, \alpha]^k$, the points of the orbit $(g'^i_a(P_{\infty\, a}))_{i\ge 1}$ are $\theta$-distant to $P_{\infty\, a}$. With the notation of the proof above, for every $j\ge 1$, we define inductively:
\[ P_{j\, a}^0:= P_{j\, a},\quad P_{j\, a}^i:= (g_a|\tilde Y_{\delta^a_{i}})^{-1}( P_{j\, a}^{i+1}), \quad -j\le i<0,\]
and $P_{j\, a}^i:= (g_a|W^u_{3/4^{d'+1}}(\Omega_a;g_a))^{i+j}(P_{j\, a}^{-j})$ for every $i\le -j$.  We notice that $(P_{j\, a}^i)_{i\le 0}$ is a preobit $P_{j\, a}$ which satisfies that: 
\begin{itemize}
\item $W^u_{loc}(P_{j\, a}; g_{a})= W^u_{loc}((P_{j\, a}^i)_{i\le 0}; g_{a})$,
\item $(P_{j\, a}^i)_{i\le -1}$ is $\theta$-distant to $P_{\infty \, a}$ and $P_{j\, a}$.
\end{itemize} 
On the other hand, the points of the orbit $ (g_a^j(P_{\infty \, a}))_{j\ge 1}$ are $\theta$-distant to $P_{\infty \, a}$.

By Proposition \ref{cartecool4cpct}, the family $(W^u_{loc} (\underline \delta; g_{a}))_a$ depends $C^{d,r}$-continuously on $\underline \delta$. In particular the modulus of continuity of $\partial_a^d W^u_{loc} (\underline \delta; g_{a})$ depends continuously on $\underline \delta$ and $a$. 
By convergence of $((W^u_{loc}(P_{j \, a}; g_{a'}))_{a'})_j $ to $(W^u_{loc} (\underline \delta^a; g_{a'}))_{a'}$ in the $C^{d,r}$-topology, it comes the following:
\begin{fact}
For every $(g_a)_a\in U_\epsilon$, there exist $U>0$ and a modulus of continuity $\nu$ so that for every $a\in [-\alpha, \alpha]^k$, the hyperbolic fixed point $\Omega(g_a)$ has a $(\theta, U,\nu)$-quasi-homoclinic $C^d$-paratangency.
\end{fact}
By Proposition \ref{quasisinkcreation}, it comes:
\begin{fact}\label{pourinduction} 
There exist $\alpha>0$,  a $C^r$-neighborhood $U_0$ of $\mathring g$, and for every $\epsilon>0$ small, there exists a neighborhood $U_\epsilon$ of 
$\{(g_{\epsilon\, a})_a:\; g\in U_0\}$ in the $C^{d,r}$-topology, so that for every $(g'_{ a})_a \in U_\epsilon$, 
for every $\mu>0$, for every large $M$ and large $N$, for every $a_0\in [-\alpha,\alpha]^k$, there exists  $(g''_a)_a\in U_\epsilon$ which satisfies:
\begin{itemize}
\item $(g''_a)_a$ is $\mu$-$C^{d,r}$-close to   $(g'_a)_a\in U_\epsilon$,
\item $g''_a=g'_a$ for every $a$ such that $(a-a_0)\notin (-2^{-N},2^{-N})^k$,
\item $g''_a$ has a periodic sink of period $M$ for every $a$ such that $(a-a_0)\in (-2^{-N-1},2^{-N-1})^k$.
\end{itemize}
 \end{fact}
\begin{rema}\label{rema2}
All the sinks have their orbits which do not intersect $[3.5,4.5]\times \R$.
\end{rema}
Applying this fact to every $a_0\in [-\alpha,\alpha]^k\cap 2^{-N+1} (\Z^k\setminus \{0\})$, the persistence of the periodic sinks shows the first step of the following induction hypothesis.

\paragraph{Induction hypothesis} 
 There exists an open and dense subset $O^n_\epsilon$ in $U_\epsilon$, so that for every $(g_a)_a\in O^n_\epsilon$, there exist increasing sequences $(N_i)_{n\ge i\ge 0}$ and $(M_i)_{n\ge i\ge 0}$ satisfying that for every $a\in [-\alpha,  \alpha]^k\cap \alpha\cdot 2^{-N_i+1} (\Z^{k}\setminus \{0\}+[-1/4, 1/4]^k)$, the map $g_a$ has a periodic sink of period at least $M_i$. Moreover, for every $g\in U_0$, the set $\{(g_a)_a\in O^n_\epsilon:\; g_0=g\}$ is open and dense in $\{(g_a)_a\in U_\epsilon:\; g_0=g\}$.

Note that to push one step forward the induction hypothesis, it suffices to reapply Fact \ref{pourinduction}.

We notice that the Lebesgue measure of $[-\alpha,  \alpha]^k\cap \alpha\cdot 2^{-N_i+1} (\Z^{k}\setminus \{0\}+[-1/4, 1/4]^k)$ is equivalent to $2^{-k}\alpha^k$ when $N_i$ is large. Hence the series $\sum_{i\ge 0} \leb([-\alpha,  \alpha]^k\cap \alpha\cdot 2^{-N_i+1} (\Z^{k}\setminus \{0\}+[-1/4, 1/4]^k))$ is absolutely divergent.  
Moreover these Borel subsets are independent. Thus, by the second Borel-Cantelli Lemma, the set 
$\cap_{j\ge 0} \cup_{i\ge j} [-\alpha,  \alpha]^k\cap \alpha\cdot 2^{-N_i+1} (\Z^{k}\setminus \{0\}+[-1/4, 1/4]^k)$ contains Lebesgue almost every point in $[-\alpha,  \alpha]^k$. This shows that for every $(g_a)_a\in \mathcal R_\epsilon:=\cap_n O^n_\epsilon$, for Lebesgue almost every $a\in [-\alpha,\alpha]^k$, 
the map $g_a$ has infinitely many sinks, with orbit in $(\mathbb Z/6\Z\setminus [3.5,4.5])\times\R$.
The set $\mathcal R:=\cup_{n>1/\epsilon} \mathcal R_{1/n}$ is Baire residual in the open set $U:= \cup_{n>1/\epsilon} U_{1/n}$. 
 Furthermore, the set $U$ has its closure which contains the constant families $(g)_a$ for every $g\in U_0$. Also for every $a\in [-\alpha,\alpha]^k$ 	and $(g_a)_a\in U$, the map $g_a$ has infinitely many sinks and sends the annulus $\mathbb S^1\times[-2,2]$ into $\mathbb S^1\times[-7/4,7/4]$. 

Also we notice that for every $g\in U_0$, the set $\{(g_a)_a\in \mathcal R:\; g_0=g\}$ is Baire residual in the open set $\{(g_a)_a\in U:\; g_0=g\}$.

This is almost all the properties of Theorem \ref{main}: we actually want that for every $(g_a)_a\in\mathcal R$, for every $a\in [-\alpha,\alpha]^k\setminus\{0\}$ (and not almost every $a$), the map $g_a$ has infinitely many sinks.

To get this extra conclusion, we shall work with $2^k$-copies of the map $\mathring g$ in a same annulus. By modifying $\mathring g$ on $\mathbb S^1\times([-3,3]\setminus [-2,2]$, we can suppose that it sends $\mathbb S^1\times\{\pm 3\}$ onto them self, and leaves invariant each of these circles.
Now consider the map:
\[\mathring G(x,y)\in \mathbb S^1\times[-3,-3+6\cdot 2^k]\mapsto \mathring g(x,y-6j)+6j, \quad \text{if}\; y\in [-3+6j ,3+6j], \; j\in \{0,\dots,2^k-1\}.\]
For $C^r$-perturbation $G$ of $\mathring G$, we define similarly $ G_{\epsilon \, a}$.

 Now, Fact \ref{pourinduction} can be applied for each of the $2^k$-copies of $\mathring g$ in   $\mathring  G$ with the same constants, but instead of taking the same $a_0$ to each of the $j^{th}$-copies, we take $a_0+2^{-N}(i_j)_{i=1}^k$ with $j=\sum_{i=1}^k i_j 2^i$ and $i_j\in \{0,1\}$, with the case $a_0=0$ added for $j\not=0$.

This defines similarly a neighborhood $\hat U_0$, and for $\epsilon>0$ small, a neighborhood $\hat U_\epsilon$ of  $\{(G_{\epsilon\, a})_a:\; G\in \hat U_0\}$, which satisfies the following induction hypothesis on $m\ge 1$.
 
\paragraph{Induction hypothesis} There exist sequences of  an open and dense subset $\hat O^m_\epsilon$ in $\hat U_\epsilon$, so that for every $(G_a)_a\in \hat O^m_\epsilon$, so that for every $a_0\in [-\alpha,\alpha]^k\setminus [-\alpha /m,\alpha/m]^k$, the map $G_a$ has a periodic sink of period at least $m$.
Moreover, for every $G\in \hat U_0$, the set $\{(G_a)_a\in \hat O_\epsilon ^m:\; G_0= G\}$ is open and dense in $\{(G_a)_a\in \hat U_\epsilon:\; G_0= G\}$.
 
The main Theorem is then satisfied  with $\hat U_N:= \{(f_{\alpha\cdot a})_a:\; (f_a)_a \in \hat O_\epsilon ^N\}$ and $\hat U:= \{(f_{\alpha\cdot a})_a:\; (f_a)_a \in \hat U_\epsilon\}$.

\section{Proof of Theorem \ref{theo2}: the diffeomorphism case}\label{proof:theo2}
For this end it suffices to show the following result, with $B$ the unit ball of $\R^n$ and $B'\Subset B$:

\begin{theoprime}\label{main'}
For all $\infty> r>d\ge 0$ or $\infty> r=d\ge 2$, for all $k\ge 0$, there exist a nonempty open set $V$ in $Diff^r (B,B')$, an open set $\hat V$ in $C^d(\R^k,Diff^r (B,B'))$, and a family of open dense subsets $(\hat V_N)_{N\ge 1}$ in $\hat V$ satisfying that:
\begin{itemize}
\item  for every $(f_a)_a\in \hat{V}_N$, for every $a\in [-1,1]^k \setminus [-1/N,1/N]^k$, the map $f_a$ has at least $N$-sinks,
\item for every $f\in V$, the set $\{(f_a)_a\in \hat V_N:\; f_0=f\}$ is open and dense in $\{(f_a)_a\in \hat V:\; f_0=f\}$.
\end{itemize}\end{theoprime}

Indeed, for every $n$-manifold, we can fix two balls  $B'\Subset B\Subset M$ and consider the open set $W$ formed by diffeomorphisms whose restrictions to $B$ are in $V$. This set is open, and non-empty since we can extend by the identity outside of a small neighborhood of $cl(B)$ any diffeomorphisms of $V$. 
Let $\hat W$ and $\hat W_n$ be the subsets of $C^d((-1,1)^k, Diff^r(M))$ formed by families of diffeomorphisms whose restrictions to $B'$ are equal to those of the elements in $\hat V$ and $\hat V_n$ respectively. These subsets are open, and $\hat W_n$ is dense in $\hat W$. Note that $\mathcal R:= \cap_n \hat W_n$ is Baire residual in $\hat W$ and satisfies Theorems \ref{theo2} when $M$ is of dimension $n$.

\begin{proof}[Proof of Theorem \ref{main'}']
For the sake of simplicity, we assume here the annulus $\mathbb A$ equal to $\mathbb S^1 \times[-1,1]$ with $\mathbb S^1 := \R/6 \Z$. 

By Theorem \ref{main} and remarks \ref{rema1} and \ref{rema2}, for all $\infty> r>d\ge 0$ or $\infty> r=d\ge 2$, for all $k\ge 0$, there exist $\alpha>0$, an open interval $I\subset \mathbb S^1$, a nonempty open set $U$ in $C^r (\mathbb A,\mathbb A)$ formed by local diffeomorphisms, an open set  $\hat U$ in $C^d(\R^k, C^r(\mathbb A,\mathbb A))$, and a family of open dense subsets $(\hat U_N)_N$ in $\hat U$  satisfying that:
\begin{enumerate}[$(i)$]
\item  For every $(f_a)_a\in \hat{U}_N$ for every $a\in[-1,1]^k\setminus [-\frac1N, \frac1N]^k$, the map $f_a$ has $N$-sinks.
\item  For every $f\in U\cup \bigcup_{\hat U}\{f_a:a\}$, for every $\theta \in \mathbb S^1$, the map $f|\{\theta\} \times[-1,1]$ is injective.  
\item  All the sinks have their orbits outside of $I\times[-1,1]$.
\item For every $g\in U$, the set $\{(f_a)_a\in {\hat U}_N:\; f_0=g\}$ is open and dense in $\{(f_a)_a\in {\hat U}:\; f_0=g\}$. 
\end{enumerate}

Let $J:=(-2,2)$.


Let $\rho\in C^\infty(\mathbb A, J^{n-2})$ and $\lambda>0$ be small. For $f\in  U\cup \{f_a\colon (f_a)\in \hat U,\; a\in \R^k\}$, 
 we consider the map:
\[\tilde f\colon \mathbb A\times J^{n-2}\ni (z,h)\mapsto (f(z), \lambda h+\rho(z))\in \mathbb A\times J^{n-2}.\]

The map $\tilde f$ is still a local diffeomorphism. The bundle $\mathbb A\times \R^{n-2}\to\mathbb A$ is left invariant by $\tilde f$. The action induced by $\tilde f$ on the base  is $f$.

 As the fibers of this bundle are much more contracted than the dynamics $f$ of the base, this invariant fibration is persistent for perturbations of the dynamics $\tilde f$:

\begin{lemm}\label{fibration}
For every $\infty>r>d\ge 0$ or $\infty> r=d\ge 2$, for every $\lambda$ small enough, for every $f$ of class $C^{r+1}$, for every $C^{r+1}$-perturbation $\tilde g$ of $\tilde f$, there exist a neighborhood $\mathbb A'\subset \mathbb A$ of $f(\mathbb A)$, a $C^r$-map $\pi_{\tilde g}\colon \mathbb A'\times J^{n-2}\to \mathbb A$ and a $C^r$-map $g\in C^r(\mathbb A, \mathbb A)$ which is $C^r$-close to $f$ and so that:
\[\pi_{\tilde g} \circ \tilde g =g\circ \pi_{\tilde g}.\] 

Moreover, if $(f_a)_a\in \hat U$ satisfies that $(a,z)\mapsto f_a(z)$ is of class $C^{r+d+1}$, and $(a,z)\mapsto \tilde g_a(z)$ is $C^{r+d+1}$-close to $(a,z)\mapsto \tilde f_a(z)$, then $(\pi_{\tilde g_a})_a$ is of class $C^{r,d}$ and $(g_a)_a$ is $C^{r,d}$-close to $(f_a)_a$.
\end{lemm}
\begin{proof}
Let us show the first statement of this lemma. We will see that the second part is similar.

Let $\tilde g$ be as in the above statement. The differential $D_z\tilde g$ acts canonically on the Grassmannian $Gr(n-2,n)$ of $(n-2)$-planes of $\R^n$ as a map denoted by $[D_z\tilde g]$. We notice that $[D_z\tilde f]$ fixes the point $\{0,0\}\times \R^{n-2}\in Gr(n-2,n)$. 
The differential $D_{\{0,0\}\times \R^{n-2}} [D_z\tilde f]$ is expanding by a factor at least $\Lambda$ large when $\lambda$ is small. We suppose the perturbation small enough so that the same holds for $D_{\{0,0\}\times \R^{n-2}} [D_z\tilde g]$. 

We notice that the manifold $N_0:=\mathbb A\times J^{n-2}\times \{(0,0)\times \R^{n-2}\}$ of $\mathbb A\times J^{n-2}\times Gr(n-2,n)$ is left invariant by the $C^r$-map:
\[\tilde F\colon (z,P)\mapsto (\tilde f(z),[D_z\tilde f](P)).\]
We notice that $D\tilde F|TN_0=(D\tilde f|T(\mathbb A\times J^{n-2}),0)$. Hence the action of $D\tilde f$ on the tangent bundle is dominated by a certain constant $K$ independent of $\lambda$. 
On the other hand the action of $D\tilde f$ on the normal bundle is expanded by a factor at least $\Lambda$.  Hence for $\lambda $ small enough, the submanifold $N_0$ is $r$-normally expanded by $\tilde F$ (see definition in \cite{berlam}). 

For $\tilde g$ $C^{r+1}$-close to $\tilde f$, we regard the following map.  
\[\tilde G\colon (z,P)\mapsto (\tilde g(z),[D_z\tilde g](P)).\]
By Theorem 3.1 of \cite{berlam}, there exists a manifold $N_{\tilde g}$ which is $C^r$-close to $N_0$ and left invariant by $\tilde G$. In particular, $N_{\tilde g}$ is the graph of a map $E^s_{\tilde g}\in C^r(\mathbb A\times J^{n-2}, Gr(n-2,n))$ which is $C^r$-close to the constant map $z\mapsto \{0,0\}\times \R^{n-2}$.
By invariance, we have:
\[D\tilde g(E^s_{\tilde g}(z))=E^s_{\tilde g}(\tilde g(z)).\]
The field $E^s_{\tilde g}$ is actually the strong stable direction. It is well known that this direction is integrable (see \cite{Yoyointro}). Let $\mathcal F_{\tilde g}$ be the foliation integrating $E^s_{\tilde g}$. This foliation is $C^r$-close to the one defined by the fibration $\mathbb A\times J^{n-2} \to \mathbb A$. 

We recall that $f(\mathbb A)\Subset \mathbb A$. Let $\mathbb A'$ be a compact neighborhood of $f(\mathbb A)$ included in the interior of $\mathbb A$. 
Consequently, for $\tilde g$ close enough to $\tilde f$, it holds:
\[\tilde g(\mathbb A\times J^{n-2})\subset \mathbb A'\times J^{n-2}.\]
Also, for $\tilde g$ close enough to $\tilde f$, for every $z\in \mathbb A'\times J^{n-2}$, the $\mathcal F_{\tilde g}$-leaf of $z$ intersects $\mathbb A\times\{0\}$ at a unique point $\pi_{\tilde g}(z)$. 

The map $\pi_{\tilde g}$ is of class $C^r$ and $C^r$-close to the first coordinate projection.

Let $g\colon z\in \mathbb A \mapsto \pi_{\tilde g}\tilde g(z,0)\in \mathbb A$. We notice that $g$ is $C^r$-close to $f$. Clearly it holds:
\[\pi_{\tilde g} \circ \tilde g =g\circ \pi_{\tilde g}.\] 

For the second part of the Lemma, we consider $(f_a)_a\in \hat U$. We notice that the submanifold $N_0\times (-1,1)^k$ is left invariant and $r$-normally expanded by the map:
\[  (z,P,a)\in (\mathbb A\times J^{n-2})\times Gr(n-2,n) \times (-1,1)^k\mapsto (\tilde f(z),D\tilde f(u), a)\in  (\mathbb A\times J^{n-2})\times Gr(n-2,n)\times (-1,1)^k.\]
Then we conclude similarly.
\end{proof}

\begin{rema}
Since the unstable manifolds of the blender are robustly tangent to the stable manifolds of the saddle points, the same holds with their preimages by the fibration. 
Note that the preimages of the stable manifolds by the fibration are still stable manifolds for the diffeomorphism, whereas the preimages of the unstable manifolds only contain the corresponding one-dimensional unstable manifolds of the diffeomorphism.
\end{rema}

Let $U'\Subset U$ and $\hat U'\subset \hat U$ so that $cl(\hat U'\cup \{(f)_a\colon f\in U'\})\Subset cl(\hat U\cup \{(f)_a\colon f\in U\})$. Let $V$ be a small neighborhood of $\{\tilde f\colon f\in U\}$ in $C^r(\mathbb A \times J^{n-2}, \mathbb A \times J^{n-2})$ and let $\hat V$ be a small neighborhood of $\{(\tilde f_a)_a\colon (f_a)_a\in \hat U\}$ in $C^d((-1,1)^k, C^r(\mathbb A \times J^{n-2}, \mathbb A \times J^{n-2}))$ so that:
\begin{itemize}
\item The closure of $\hat V$ contains $\{(\tilde g)_a\colon \tilde g \in V\}$, and Lemma \ref{fibration} defines a  family $(g_a)_a$ which belongs to $\hat U$. 
\end{itemize}
Hence by Theorem \ref{main}, for every $N\ge 0$, there exists a $(g'_a)_a$ which is $C^{d,r}$-close to $(g_a)_a$ satisfying $g'_0=g_0$ and such that for every $a\in [-1,1]^k\setminus[-1/N,1/N]^k$, the map $g'_a$ has $N$ different attracting cycles. For every $z\in \mathbb A\times J^{n-2}$, let $\tau_a(z)$ be the vector in $\R^2\times \{0\}$ such that:
\[\pi_{\tilde g_a}(\tilde g_a(z)+\tau_a(z))= g'_a\circ \pi_{\tilde g_a}(z)\]
Note that $\tau_0=0$. 
We remark also that the map $\tilde g'_a:= \tilde g_a+\tau_a$ is of class $C^r$, and that the family $(\tilde g'_a)_a$ is of class $C^{d,r}$ and $C^{d,r}$-close to $(\tilde g_a)_a$. Moreover we have $\pi_{\tilde g'_a}=\pi_{\tilde g_a}$ and:
\[\pi_{\tilde g_a'}\circ \tilde g_a'= g_a'\circ \pi_{\tilde g_a'}.\]
Consequently the map $\tilde g'_a$ has at least $N$ different attracting cycle for every $a\in [-1,1]^k$. 
Also $\tilde g_0'=\tilde g_0$. By persistence of the attracting cycles, this property holds true for families in a $C^{d,r}$-neighborhood of $(\tilde g_a')_a$. Taking the union over $(\tilde g_a)_a\in \hat V$, this define an open and dense subset $\hat V_N$ in $\hat V$, such that for every $(\tilde g_a')_a\in \hat V_N$ and every $a\in [-1,1]^k\setminus[-\frac1N, \frac1N]^k$, the map $\tilde  g_a$ has $N$-different cycles.
Furthermore, for every $\tilde g\in V$, the set $\{(\tilde g_a)_a\in \hat V_N:\; \tilde g_0 =g\}$ is open and dense in $\{(\tilde g_a)_a\in \hat V:\; \tilde g_0=g\}$. 

Now it remains to chose $\rho$ so that $\tilde g$ is invertible, and so a diffeomorphism. 

\paragraph{Invertibility in dimension $n\ge 4$.} 
Let $\rho\colon  (\theta,y)\in \mathbb A\mapsto (\cos(2\pi \theta/6), \sin (2\pi \theta/6), 0,\cdots,0)\in \R^n$. 
By $(iii)$, for every $f\in U\cup \cup_{\hat U}\{f_a:a\}$, for every $\theta \in \mathbb S^1$, the map $f|\{\theta\} \times[-1,1]$ is injective.  Consequently, the following map is injective, and so is an embedding:
\[(\theta,y)\in \mathbb A\mapsto (g(x,y), \rho(\theta))\in \mathbb A \times J^{n-2}.\]
Consequently, for such a $\rho$ and with $\lambda>0$ small enough, the map $\tilde f$ is a diffeomorphism from $\mathbb A\times J^{n-2}$ into itself. It is when known from the Smale solenoid construction that it is possible to extend $\tilde g$  to a diffeomorphism of $\R^n$ equal to the identity outside of a ball.
 
\paragraph{Invertibility in dimension $n= 3$.} We recall that $\mathbb A= \mathbb S^1 \times [-1,1]$.
For this dimension, we use not only $(iii)$ but also $(iv)$: for every $f\in U\cup \bigcup_{\hat U}\{f_a:a\}$, all the sinks have their orbit outside of $I\times[-1,1]$, and the map $f|I\times[-1,1]$ is close to $(\theta,y)\mapsto (Q(\theta)^{d+1},2y/3)$. 
Assume that the interval $I$ is $[-\eta 4^{-d'-1},\eta 4^{-d'-1}]$, for $\eta$ small. Let $I':= [-\eta 4^{-d'-2},\eta 4^{-d'-2}]$.

Let $d\rho\in C^\infty(\R/6\Z, (-\infty, 1/6])$ be equal to $1/6$ on the complement of  $I'$ and with  mean equal to $0$. Let $\rho\in C^\infty (\R/6\Z, \R) $ be the function integrating $d\theta$ and equal to $0$ at $3\in \mathbb R/6\Z$.  We notice that $\rho$ has its values in $[-1/2,1/2]$. Hence the following map is an immersion:
\[(\theta,y)\in \mathbb A\mapsto (g(\theta,y), \rho(\theta))\in \mathbb A \times \R,\]
and by $(iii)$, its restriction to $\mathbb A\setminus I'\times[-1,1]$ is injective. To make it injective on the whole annulus, we shall modify it in the $y$ direction. Let $\phi\in C^\infty (\mathbb A, [0,3])$ be a function equal to $0$ on the complement of $I\times[-1,1]$ and equal to $3$ on $I' \times [-1,1]$. We remark that the following map is injective and so an embedding:
\[(\theta,y)\in \mathbb A\mapsto (g(\theta,y), \rho(\theta))+(0,\phi(\theta), 0)\in \R/6\Z\times[-1,4]\times \R.\]
For $\mu >0$ small, we can extend the map $g$ to $\R/6\Z\times[-2,5]$, so that $g$ is still a local diffeomorphism and its image is in $\R/6\Z\times [-1-\mu ,1+\mu]$. Then we notice that the following map is an embedding:
\[\breve g\colon (\theta,y)\in \R/6\Z\times[-2,5] \mapsto (g(\theta,y), \rho(\theta))+(0,\phi(\theta), 0)\in \R/6\Z\times[-2,5]\times [-1/2,1/2].\] 
Also, for $\lambda>0$ small enough,  the following map is a diffeomorphism onto its image:
\[\tilde g'\colon (\theta,y)\in \R/6\Z\times[-2,5]\times J\mapsto \breve g (\theta,y)+(0,0,\lambda y)\in  \R/6\Z\times[-2,5]\times J.\]
Again, similarly to the Smale solenoid construction that it is possible to extend $\tilde g'$  to a map of $\R^3$ equal to the identity outside of a ball.
As the map $\tilde g'$ and $\tilde g$ are equal on $(\R/6\Z\setminus I)\times [-1,1]\times J$, by $(iv)$, the existence of sinks for these perturbations remains valid.
\end{proof}

\appendix
\section{Proof of Propositions in the hyperbolic dynamics preliminary}\label{section:proofhypprel}
\label{Propositioncartecoolproof}
\begin{proof}[Proof of Proposition \ref{cartecool}]

It is well known that $(W^s_\eta(\Omega_a))_a$ and $(W^u_\eta(\Omega_a))_a$ are continuous families of $C^r$-manifolds. To prove the existence of such families of charts, it suffices to show that the family of local stable and unstable manifolds $(W^s_\eta(\Omega_a))_a$ and $(W^u_\eta(\Omega_a))_a$ are of class $C^{d, r}$.
As one can obtain the proof of the regularity of the stable manifolds from the one of unstable manifolds by inverting the dynamics, we shall show only that
the family $(W^s_\eta(\Omega_a))_a$ is of class $C^{d,r}$.

We remark that the following function vanishes at the graph of $(\Omega_a)_a$:
\[(a,z)\in (-1,1)^k \times \R^n\mapsto (f_a(z)-z).\]
As its derivative with respect to $z$ is an isomorphism, by the implicit function Theorem, the family $(\Omega_a)_a$ is of class  $C^d$. 
Thus by conjugating with a translation, we can suppose that $\Omega_a=\Omega:=0$ for every $a$.

Let $E^s_a$ be the stable direction of $D_\Omega f_a$. We remark that the action of $D_\Omega f_a$ on the Grassmannian of $\dim E^s_a$-plan of $\R^n $ is contraction.
As $a\mapsto D_\Omega f_a^{-1}$ is of class $C^d$, its fixed point $E^s_a$ is a $C^d$-function of $a$. 
Thus, by a linear coordinates change, we can suppose the stable and unstable directions are independent of $a$.

To prepare the proof of Proposition \ref{cartecool4cpct}, we proof this proposition in the more geral case where $(f_a)_a\in C^d((-1,1)^k, Diff^r(E))$, with $E$ a Banach space, so that there exists an invariant splitting  $E=E^s\oplus E^u$ left invariant by $D_{\Omega}f_a$, with $E^s$ $\lambda$-contracted by $D_{\Omega}f_a$ and  $E^u$ are $\lambda$-contracted by $D_{\Omega}f_a^{-1}$, for every $a\in (-1,1)^k$. 

We now push forward the proof of the stable manifold theorem of Irwin written in \cite{Yoyointro}, where more details are given, to get the $C^d$-regularity of the family.
Let $T^s_a:= D_{\Omega}f_a|E^s$ and $T^u_a:= D_{\Omega}f_a|E^u$, and let $p_s\colon E\to E^s$ and $p_u\colon E\to E^u$ be the canonical projections. Let $f^s_a:= p_s\circ f_a$ and $f_a^u:= p_u\circ f_a$. To any sequence $(x_n)_n=(x_n^s,x_n^u )_n\in E^\mathbb N$, we associate the following sequences $(y_n^s)_{n\ge 1}$ and $(y_n^u)_{n\ge 0}$:
\[\left\{\begin{array}{l} y^s_{n+1}= f^s_a(x^s_n, x^u_n)\\
y^u_n= x_n^u+ (T^u_a)^{-1}(x_{n+1}^u-f^u_a(x_{n}^s,x_n^u))\end{array}\right.\]
Observe that 
\begin{equation}\label{implicit}
 f(x_n)=x_{n+1}\Leftrightarrow \left \{\begin{array}{l} x_{n+1}^s= y_{n+1}^s\\ x_{n+1}^u = y_{n+1}^u\end{array}\right.\end{equation} 

We remark that for every $\kappa\in (\lambda,1)$, if $\eta>0$ is sufficiently small, we consider the spaces:
\[\mathcal E:= \{(x_n)_{n\ge 0}\in {E}^{\mathbb N+}\colon \sup_n \kappa^{-n} |x_n|<\infty\}\]
\[\mathcal E^s:= \{(x^s_n)_{n\ge 1}\in {E^s}^{\mathbb N^*}\colon \|(x_n^s)\|_{\mathcal E^s }:= \sup_n \kappa^{-n} |x_n^s|<\infty\}\]
\[\mathcal E^u:= \{(x^u_n)_{n\ge 0}\in {E^u}^{\mathbb N}\colon \|(x_n^u)\|_{\mathcal E^u }:= \sup_n \kappa^{-n} |x_n^u|<\infty\}\]
We endow $E^s$ with the norm $|\cdot |$ induced by the one of $E$. Let $E^s_\eta$ be the $\eta$-ball centered at $0$ of $E^s$.
We equip $ \mathcal E^s\times \mathcal E^u$ and  $\mathcal E=E^s\times \mathcal E^s\times \mathcal E^u$ with the sup norms $\max(\|\cdot \|_{\mathcal E^s},\|\cdot \|_{\mathcal E^u})$ and $\max(|\cdot |, \|\cdot \|_{\mathcal E^s},\|\cdot \|_{\mathcal E^u})$.

From \cite[\textsection 3.2]{Yoyointro}, the following map is well defined at the neighborhood of $0\in \mathcal E$:
\[\theta_a \colon (x_0,(x_n)_{n\ge 1}, (y_n)_{n\ge 0})\in \mathcal E\mapsto ((y_n^s)_n,(y_n^u)_n)\in   \mathcal E^s\times \mathcal E^u\; .\]
Moreover, for every $x^s\in E^s_\eta$, the map $\theta_a(x^s,\cdot)$ is contracting \cite[\textsection 3.2]{Yoyointro}. Let $\phi_a(x^s):= ((\phi_n^s(x^s))_{n\ge 1},(\phi_n^u(x)_{n\ge 0})\in   \mathcal E^s\times \mathcal E^u$ be its fixed point. Let us regard the map:
\[g_a\colon x^s\in E^s_\eta\mapsto \phi_0^u(x^s)\in E^u.\]  
By (\ref{implicit}), the graph of $g_a$ denoted by $W^s_\eta(\Omega_a; f_a)$, is left invariant by $f_a$.

To achieve the proof of the proposition, it suffices to prove that $(a,x^s)\mapsto g_a(x^s)$ has all its derivatives of the form $\partial_a^i\partial^j_{x^s} g_a$, for all $i\le d$ and $j\le r$, which are well defined and continuous. By using the implicit function theorem, it suffices to prove that the following function is $C^1$ and that its derivative of the form $\partial_a^i\partial^j_{x^s_0} $ are well  defined and continuous for all $i\le d$ and $j\le r$:
\[\Phi\colon
(-1,1)^k\times E^s\times \mathcal E^s\times \mathcal E^u\to \mathcal E^s\times \mathcal E^u \]
\[(a, x^s_0,(x^s_n)_{n\ge 1},(x^u_n)_{n\ge 0})\mapsto ((x^s_n-y^s_n)_{n\ge 1},(x^u_n-y^u_n)_{n\ge 0}),\]
\[\text{with}\quad x^s_n-y^s_n=x_n^s-f^s_a(x^s_{n-1}, x^u_{n-1}), \quad\text{and}\quad  x^u_n-y^u_n=(T^u_a)^{-1}(x_{n+1}^u-f^u_a(x_{n}^s,x_n^u).\]
Clearly, $\Phi$ has these derivatives well defined and continuous if and only if this holds for the map:
\[F\colon (a,(x_n)_n)\in (-1,1)^k\times \mathcal E\mapsto (f_a(x_n))_n\in \mathcal E.\]
It is shown in \cite[\textsection 3.3]{Yoyointro} that the derivatives $\partial_{\mathcal E}^j F$ exists and are continuous for every $i\in \{0,\cdots, r\}$. Moreover, it holds:
\[\partial_{\mathcal E}^j F\left(a,(x_n)_{n\ge 0}\right)\colon \left((\Delta x_1^n)_{n\ge 0}, \dots, (\Delta x_j^n)_{n\ge 0}\right)\in \mathcal E^j\mapsto \left(\partial_x^j f_a(x_n)(\Delta x_1^n, \dots, \Delta x_j^n)\right)_{n\ge 0}\in \mathcal E\]
Let us now study the differentiability with respect with the parameter $a$ of $\partial_{\mathcal E}^j F$ at $(x_n)_n\in \mathcal E$.

Let $\Delta a$ be small. For $n\ge 0$, $t\in [0,1]$, $i\le d$, we put:
\[D_{i,j}(t,n):= (\Delta x_1^n, \dots, \Delta x_j^n) \mapsto \partial^i_a\partial^j_x f_{a_t}(x^n)(\Delta x_1^n, \dots, \Delta x_j^n),\quad a_t:= a+t\Delta a.\]  
We have:
\[\partial^j_x f_{a+\Delta a}(x_n)(\Delta x_1^n, \dots, \Delta x_j^n)= \left(\sum_{l=0}^{i}  \frac{1}{l!}D_{l,j}(0,n)(\Delta a)^l+ R_{i,j}\right)(\Delta x_1^n, \dots, \Delta x_j^n),\]
\[R_{i,j}:= \frac{1}{(i-1)!}\int_0^1 (1-t)^{i-1}(D_{i,j}(t,n)-D_{i,j}(0,n)) (\Delta a)^i dt. \]
Because $\kappa <1$, we have $\lim_{n}|x_n|=0$. Hence $D_{i,j}=\sup_{n}\|D_{i,j}(t,n)\|<\infty$. Also the $(l,j)$-linear map:
\[\Delta_{i,j}\colon (\Delta a_1,\dots , \Delta a_l,\Delta x_1^n, \dots, \Delta x_j^n) \mapsto D_{l,j}(0,n)(\Delta a_1,\dots , \Delta a_l)(\Delta x_1^n, \dots, \Delta x_j^n)\]
sends $\mathbb R^l\times \mathcal E^l$ into $\mathcal E$ and has norm $\le D_{l,j}$. From this we conclude easily that $(\partial^i_a\partial_{\mathcal E}^j F)$ exists continuously, is equal to $\Delta_{i,j}$, and for all $i\le d$ and $j\le r$.
\end{proof}
\begin{proof}[Proof of propostion \ref{cartecool4cpct}]

As for Proposition \ref{cartecool}, it suffices to show that $(W^s_\eta(\underline k_a; f_a))_a$ and $(W^u_\eta(\underline k_a; f_a))_a$ are $C^d$-families of $C^r$-manifolds to get the existence of the charts.

Let $\exp$ be the exponential map associated to a Riemannian metric on $M$.

Let $E$ be the space of continuous sections $\sigma$  of the bundle $\pi^* TM$: in particular $\sigma(\underline k)\in T_{k_0} M$, for every $\underline k\in \overleftarrow K$.
We consider the map
\[F^0_a\colon   \sigma\in E\mapsto \left[\underline k\mapsto \exp_{k_0}^{-1} \circ f_a \circ \exp_{k_{-1}} \sigma(\overleftarrow f^{-1} (\underline k))\right] \in E \] 
For $a$ small, the map $F_a^0$ is well defined and continuous on a neighborhood $N$ of the $0$ section. Also $a\mapsto F^0_a\in C^r(N,E)$ is of class $C^d$. Moreover, the zero section is a hyperbolic fixed point of $F^0_a$. Hence the implicit theorem applied to the map $F^0_a-\pi$ gives that the hyperbolic continuation 
\[a\mapsto (\underline k\mapsto i_{f_a}(\underline k)\in M)\]
is of class $C^d$. 
Let $k_0(a):= \pi(i_{f_a}(\underline k))$, for every $a$ and $\underline k\in \overleftarrow K$. Put: 
\[F_a\colon   \sigma\in E\mapsto \left[\underline k\mapsto \exp_{k_0(a)}^{-1} \circ f_a \circ \exp_{k_{-1}(a)} \sigma(\overleftarrow f^{-1} (\underline k))\right] \in E \] 
We remark that there exists a neighborhood $N$ of the zero section so that $a\mapsto F_a\in C^r(N,E)$ is of class $C^d$. Moreover, the zero section is a hyperbolic fixed point of $F_a$ for every $a$.

Let $Gr(TM)\to M$ be the Grassmannian bundle of $M$, with fiber at $z$ the Grassmannian of $Gr(T_zM)$ of $T_zM$. As $f$ is a local diffeomorphism, the map $Df$ acts canonically on $Gr(TM)$ over $f$.

We remark that the family of stable directions $(E^s_{\underline k}(f_a))_{\underline k}$ is a continuous section of $\pi^*Gr(TM)$. Let $G$ be the space of sections of the bundle $\pi^*Gr(TM)$. It is actually a contracting fixed point of the map which associate to a section $\gamma$ of  $\pi^*Gr(TM)$, the section:
\[\gamma\mapsto\left[\underline k\mapsto D_{k_{0}} f_a^{-1}\circ \gamma (\overleftarrow f(\underline k))\right].\]
As the dependence on $a$ is of class $C^d$, the section $(E^s_{\underline k}(f_a))_{\underline k}$ is a $C^d$-function of $a$. Thus, via a linear conjugacy, depending continuously on $\underline k$ and of class $C^d$ on $a$, we can suppose that $E^s_{\underline k}(f_a)=E^s_{\underline k}$ does not depend on $a$, and likewise for 
$E^u_{\underline k}(f_a)=E^u_{\underline k}$. 

Consequently, we are now in the setting to apply the proof of Proposition \ref{cartecool}: by Proposition \ref{cartecool}, there exists a continuous family of $C^{d,r}$-maps $((g_a(\underline k))_a)_{\underline k\in \overleftarrow K}$, with $g_a$ from $E^s_{\eta,\; \underline k}(f_a)$ to  $E^u_{\underline k}(f_a)$, whose graph is sent by the map $\exp_{k_0}$ to a curve $W^s_\eta(\underline k; f_a)$. Moreover the curve $W^s_\eta(\underline k; f_a)$ is sent by $f_a$ into $W^s_\eta(\overleftarrow f_0(\underline k); f_a)$ and is exponentially contracted by $f_a$. Similarly we obtains that the local unstable manifolds  families  $(W^u_\eta(\underline k; f_a))_a$ are $C^{d,r}$ and that these families depend continuously on $\underline k \in \overleftarrow K$.
From these two families of submanifolds, the requested charts are easily constructed.
\end{proof}

\bibliographystyle{alpha}
\bibliography{references}

\def\cprime{$'$} \def\cprime{$'$} \def\cprime{$'$}
\begin{thebibliography}{dAMY01}

\bibitem[ALdM03]{ALDM03}
A.~Avila, M.~Lyubich, and W.~de~Melo.
\newblock Regular or stochastic dynamics in real analytic families of unimodal
  maps.
\newblock {\em Invent. Math.}, 154(3):451--550, 2003.

\bibitem[AM02]{AM02}
A.~Avila and C.~G. Moreira.
\newblock Hausdorff dimension and the quadratic family.
\newblock {\em Manuscript}, 2002.

\bibitem[Asa08]{Asa08}
M.~Asaoka.
\newblock Hyperbolic sets exhibiting {$C^1$}-persistent homoclinic tangency for
  higher dimensions.
\newblock {\em Proc. Amer. Math. Soc.}, 136(2):677--686, 2008.

\bibitem[BC91]{BC2}
M.~Benedicks and L.~Carleson.
\newblock The dynamics of the {H}\'enon map.
\newblock {\em Ann. Math.}, 133:73--169, 1991.

\bibitem[BD]{BD14}
P.~Berger and R.~Dujardin.
\newblock On stability and hyperbolicity for polynomial automorphisms of {${\bf
  C}^2$}.
\newblock {\em arXiv:1409.4449}.

\bibitem[BD96]{BD96}
C.~Bonatti and L.~J. D{\'{\i}}az.
\newblock Persistent nonhyperbolic transitive diffeomorphisms.
\newblock {\em Ann. of Math. (2)}, 143(2):357--396, 1996.

\bibitem[BD99]{BD99}
C.~Bonatti and L.~D{\'{\i}}az.
\newblock Connexions h\'et\'eroclines et g\'en\'ericit\'e d'une infinit\'e de
  puits et de sources.
\newblock {\em Ann. Sci. \'Ecole Norm. Sup. (4)}, 32(1):135--150, 1999.

\bibitem[BdS15]{BeSi}
P.~Berger and J.~de~Simoi.
\newblock On the hausdorff dimension of newhouse phenomena.
\newblock {\em Annales Henri Poincar\'e}, pages 1--23, 2015.

\bibitem[Bera]{Berhen}
P.~Berger.
\newblock Abundance of non-uniformly hyperbolic {H}\'enon like endomorphisms.
\newblock {\em arXiv:0903.1473v2}.

\bibitem[Berb]{MisuRen}
P.~Berger.
\newblock {N}ormal forms and {M}isiurewicz renormalization for dissipative
  surface diffeomorphisms.
\newblock {\em arXiv:1404.2235}.

\bibitem[Berc]{Berent}
P.~Berger.
\newblock Properties of the maximal entropy measure and geometry of {H}\'enon
  attractors.
\newblock {\em arXiv:1202.2822}.

\bibitem[Ber10]{berlam}
P.~Berger.
\newblock Persistence of laminations.
\newblock {\em Bull. Braz. Math. Soc. (N.S.)}, 41(2):259--319, 2010.

\bibitem[BF15]{webpage}
P.~Berger and P.-Y. Fave.
\newblock Parablender.
\newblock {\em \url{math.sunysb.edu/~berger}}, Videos, 2015.

\bibitem[BR13]{BR13}
P.~Berger and A.~Rovella.
\newblock On the inverse limit stability of endomorphisms.
\newblock {\em Ann. Inst. H. Poincar\'e Anal. Non Lin\'eaire}, 30(3):463--475,
  2013.

\bibitem[BV01]{BV2}
M.~Benedicks and M.~Viana.
\newblock Solution of the basin problem for {H}\'enon-like attractors.
\newblock {\em Invent. Math.}, 143(2):375--434, 2001.

\bibitem[BY93]{BY}
M.~Benedicks and L.-S. Young.
\newblock Sina\u\i-{B}owen-{R}uelle measures for certain {H}\'enon maps.
\newblock {\em Invent. Math.}, 112(3):541--576, 1993.

\bibitem[dAMY01]{GuYo}
Carlos Gustavo~T. de~A.~Moreira and Jean-Christophe Yoccoz.
\newblock Stable intersections of regular {C}antor sets with large {H}ausdorff
  dimensions.
\newblock {\em Ann. of Math. (2)}, 154(1):45--96, 2001.

\bibitem[DL]{LD13}
R.~Dujardin and M.~Lyubich.
\newblock Stability and bifurcations for dissipative polynomial automorphisms
  of {${\bf C}^2$}.
\newblock {\em to appear in Invent. Math.}

\bibitem[DNP06]{DNP06}
L.~J. D{\'{\i}}az, A.~Nogueira, and E.~R. Pujals.
\newblock Heterodimensional tangencies.
\newblock {\em Nonlinearity}, 19(11):2543--2566, 2006.

\bibitem[GK07]{GoKa}
A.~Gorodetski and V.~Kaloshin.
\newblock How often surface diffeomorphisms have infinitely many sinks and
  hyperbolicity of periodic points near a homoclinic tangency.
\newblock {\em Adv. Math.}, 208(2):710--797, 2007.

\bibitem[HK90]{HK90}
F.~Hofbauer and G.~Keller.
\newblock Quadratic maps without asymptotic measure.
\newblock {\em Comm. Math. Phys.}, 127(2):319--337, 1990.

\bibitem[HK10]{HK10}
B.~R. Hunt and V.~Y. Kaloshin.
\newblock Prevalence.
\newblock In {\em Handbook of Dynamical Systems}, volume~3, pages 43--87. 2010.

\bibitem[Joh87]{Jo87}
S.~D. Johnson.
\newblock Singular measures without restrictive intervals.
\newblock {\em Comm. Math. Phys.}, 110(2):185--190, 1987.

\bibitem[KSvS07]{KSvS07}
O.~Kozlovski, W.~Shen, and S.~van Strien.
\newblock Density of hyperbolicity in dimension one.
\newblock {\em Ann. of Math. (2)}, 166(1):145--182, 2007.

\bibitem[Lyu02]{Ly02}
M.~Lyubich.
\newblock Almost every real quadratic map is either regular or stochastic.
\newblock {\em Ann. of Math. (2)}, 156(1):1--78, 2002.

\bibitem[New74]{Newhouse}
S.~E. Newhouse.
\newblock Diffeomorphisms with infinitely many sinks.
\newblock {\em Topology}, 12:9--18, 1974.

\bibitem[New80]{New80}
S.~E. Newhouse.
\newblock Lectures on dynamical systems.
\newblock In {\em Dynamical systems ({C}.{I}.{M}.{E}. {S}ummer {S}chool,
  {B}ressanone, 1978)}, volume~8 of {\em Progr. Math.}, pages 1--114.
  Birkh\"auser, Boston, Mass., 1980.

\bibitem[Pal00]{Pa95}
J.~Palis.
\newblock A global view of dynamics and a conjecture on the denseness of
  finitude of attractors.
\newblock {\em Ast\'erisque}, (261):xiii--xiv, 335--347, 2000.
\newblock G{\'e}om{\'e}trie complexe et syst{\`e}mes dynamiques (Orsay, 1995).

\bibitem[Pal05]{Pa05}
J.~Palis.
\newblock A global perspective for non-conservative dynamics.
\newblock 22(4):485--507, 2005.

\bibitem[Pal08]{Pa08}
J.~Palis.
\newblock Open questions leading to a global perspective in dynamics.
\newblock {\em Nonlinearity}, 21(4):T37--T43, 2008.

\bibitem[PS96]{PS96}
Charles Pugh and Michael Shub.
\newblock Stable ergodicity and partial hyperbolicity.
\newblock In {\em International {C}onference on {D}ynamical {S}ystems
  ({M}ontevideo, 1995)}, volume 362 of {\em Pitman Res. Notes Math. Ser.},
  pages 182--187. Longman, Harlow, 1996.

\bibitem[PS00]{PS00}
E.~R. Pujals and M.~Sambarino.
\newblock Homoclinic tangencies and hyperbolicity for surface diffeomorphisms.
\newblock {\em Ann. of Math. (2)}, 151(3):961--1023, 2000.

\bibitem[PT93]{PT93}
Jacob Palis and Floris Takens.
\newblock {\em Hyperbolicity and sensitive chaotic dynamics at homoclinic
  bifurcations}, volume~35 of {\em Cambridge Studies in Advanced Mathematics}.
\newblock Cambridge University Press, Cambridge, 1993.
\newblock Fractal dimensions and infinitely many attractors.

\bibitem[PV94]{PaVi}
J.~Palis and M.~Viana.
\newblock High dimension diffeomorphisms displaying infinitely many periodic
  attractors.
\newblock {\em Ann. of Math. (2)}, 140(1):207--250, 1994.

\bibitem[PY09]{PY09}
J.~Palis and J.-C. Yoccoz.
\newblock Non-uniformly hyperbolic horseshoes arising from bifurcations of
  {P}oincar\'e heteroclinic cycles.
\newblock {\em Publ. Math. Inst. Hautes \'Etudes Sci.}, (110):1--217, 2009.

\bibitem[TLY86]{TeYo}
Laura Tedeschini-Lalli and James~A. Yorke.
\newblock How often do simple dynamical processes have infinitely many
  coexisting sinks?
\newblock {\em Comm. Math. Phys.}, 106(4):635--657, 1986.

\bibitem[Tur05]{Tu}
D.~Turaev.
\newblock On the genericity of the newhouse phenomenon.
\newblock In {\em E{QUADIFF} 2003}. World Sci. Publ., Hackensack, NJ, 2005.

\bibitem[Yoc95]{Yoyointro}
J.-C. Yoccoz.
\newblock Introduction to hyperbolic dynamics.
\newblock In {\em Real and complex dynamical systems ({H}iller\o d, 1993)},
  volume 464 of {\em NATO Adv. Sci. Inst. Ser. C Math. Phys. Sci.}, pages
  265--291. 1995.

\end{thebibliography}

\end{document}